\documentclass[10pt,leqno]{article}

\usepackage[lite]{amsrefs}
\usepackage{latexsym,enumerate,pstricks}

\usepackage[notref, notcite]{}
\usepackage{amssymb, amsfonts, eucal,amsmath,amsthm}

\newtheorem{theorema}{Theorem}

\newtheorem{theorem}{Theorem}[section]

\newtheorem{lemma}[theorem]{Lemma}
\newtheorem{proposition}[theorem]{Proposition}
\newtheorem{remark}[theorem]{Remark}
\newtheorem{definition}[theorem]{Definition}

\newtheorem{corollary}[theorem]{Corollary}

\newcommand{\sect}[1]{\section{#1} \setcounter{equation}{0} }

\newcommand{\norm}[2]{\left\|#1\right\|_{#2}}

\newcommand{\ds}{\displaystyle}

\newcommand{\Pn}{\Poly_n}

 \newcommand{\ec}{\end{comment}}
\newcommand{\bc}{ \begin{comment}
 }
 \newcommand{\bsmall}{\begin{scriptsize} \vspace{1cm} \mbox{}\\
  $\downarrow\rule{\textwidth}{0.1mm}  \downarrow$
 \setlength{\baselineskip}{0.3mm}
 }
 \newcommand{\esmall}{\noindent $\uparrow \rule{\textwidth}{0.1mm}\uparrow$ \end{scriptsize}\vspace{1cm} }

\newcommand{\E}{{\mathcal E}}
\newcommand{\DW}{{\mathcal DW}}
\newcommand{\AW}{A^*}

 \newcommand{\I}{{\mathcal I}}

\newcommand{\ccc}{{\mathfrak  \vartheta}}

\newcommand{\D}{\delta(\Z)}

\newcommand{\andd}{\quad\mbox{\rm and}\quad}

\newcommand\e{{\varepsilon}}

\newcommand\w{{\omega}}

\newcommand{\Z}{\mathcal Z}
\newcommand{\bZ}{{\mathbb Z}}
\newcommand{\bZM}{\bZ_M}

\newcommand{\pmin}{\mathrm{pmin}}
\newcommand{\astar}{$A^*$}

\def\be  {\begin{equation}}
\def\ee  {\end{equation}}
\def\ba  {\begin{eqnarray}}
\def\ea  {\end{eqnarray}}
\def\baa {\begin{eqnarray*}}
\def\eaa {\end{eqnarray*}}
\newenvironment{comment}[2]
{\bgroup\vspace{7pt}
\begin{tabular}{|p{5in}|}
\hline \qquad \bf \footnotesize Comment -- to be deleted in the final version \\
\hline
\quad\sl\footnotesize #1#2} {\\ \hline \end{tabular}
\vspace{7pt}\indent\egroup}

\def \meas{\mathop{\rm meas}\nolimits}

\def \dist{\mathop{\rm dist}\nolimits}

\def \esssup{\mathop{\rm ess\: sup}\nolimits}

  \newcommand{\K}{{\mathit{K}}}

\newcommand{\W}{{\mathcal{W}}}

\newcommand{\R}{\mathbb R}
\newcommand{\Range}{\mathcal R}

\newcommand{\N}{\mathbb N}

\newcommand{\ineq}[1]{(\ref{#1})}

\newcommand{\ie}{{\em i.e., }}
\newcommand{\eg}{{\em e.g. }}

\newcommand{\bpic}{
\begin{center}
}

\newcommand{\epic}{
\endpspicture
\end{center}
}

\newcommand{\st}{\;\; \big| \;\;}
\renewcommand{\L}{\mathbb{L}}

\newcommand{\Poly}{\Pi}
 \newcommand{\AC}{\mathrm{AC}}
  \newcommand{\loc}{\mathrm{loc}}

\newcommand{\Dom}{{\mathfrak{D}}}
 \newcommand\J {{\mathcal{J}}}
 
\newcommand{\wj}{w_{\J}}

\newcommand{\thm}[1]{Theorem~\ref{#1}}
\newcommand{\lem}[1]{Lemma~\ref{#1}}
\newcommand{\lemp}[1]{Lemma~\ref{properties}(\ref{#1})}
\newcommand{\cor}[1]{Corollary~\ref{#1}}

\newcommand{\rem}[1]{Remark~\ref{#1}}
\newcommand{\deff}[1]{Definition~\ref{#1}}

\title{Uniform polynomial  approximation with \astar{} weights having finitely many zeros}

\author{
Kirill A.  Kopotun\thanks{Department of Mathematics, University of
Manitoba, Winnipeg, Manitoba, R3T 2N2, Canada ({\tt
kopotunk@cc.umanitoba.ca}). Supported by Natural Sciences and Engineering Research Council of Canada under grant RGPIN-2015-04215.}    }

\begin{document}

\maketitle

\begin{abstract}
We prove matching direct and inverse theorems for uniform polynomial approximation   with $A^*$ weights (a subclass of doubling weights suitable for approximation in the $\L_\infty$ norm) having finitely many   zeros  and not too ``rapidly changing'' away from these zeros. This class of weights
 is rather wide and, in particular, includes the classical Jacobi weights, generalized Jacobi weights and generalized Ditzian-Totik weights. Main part and complete weighted moduli of smoothness are introduced, their properties are investigated, and
equivalence type results involving related realization functionals are  discussed.

\end{abstract}

%


\sect{Introduction}

Recall that  a nonnegative integrable  function $w$ is a  doubling weight (on $[-1,1]$)    if there exists a positive constant $L$ (a so-called doubling constant of $w$) such that
\be \label{doub}
 w (2I)  \leq L w (I) ,
\ee
for any  interval $I\subset [-1,1]$. Here,  $2I$ denotes the interval of length $2|I|$ ($|I|$ is the length of   $I$) with the same center as $I$, and
$ w(I) := \int_{I}  w(u) du$.
 Note that it is convenient to assume that $w$ is identically zero outside $[-1,1]$ which allows us to write $w(I)$ for any interval $I$ that is not necessarily contained in $[-1,1]$.
Let  $\DW_L  $ denote  the set of all doubling weights on $[-1,1]$ with the doubling constant $L$, and
$\DW  := \cup_{L>0}  \DW_L$,
\ie $\DW$ is the set of all doubling weights.

It is easy to see that $w \in \DW_L$   if and only if there exists a constant $\kappa\geq 1$ such that,  for any two adjacent intervals $I_1, I_2 \subset [-1,1]$ of equal length,
\be \label{withkap}
w(I_1) \leq \kappa w(I_2) .
\ee
Clearly, $\kappa$ and $L$ depend on each other.
In fact, if $w \in \DW_L$ then \ineq{withkap} holds with $\kappa = L^2$. Conversely, if \ineq{withkap} holds, then  $w \in \DW_{1+\kappa}$.

Following \cites{mt1999, mt2000}, we   say that   $w$ is an  \astar{} weight (on $[-1,1]$) if there is a constant $L^*$  (a so-called    \astar{} constant of $w$)
such that, for all  intervals $I\subset [-1,1]$ and $x\in I$, we have
\be \label{astar}
w(x) \leq {L^* \over |I|}   w(I)  .
\ee

Throughout this paper,  $\AW_{L^*}$ denotes  the set of all \astar{} weights on $[-1,1]$ with the \astar{} constant $L^*$. We also let
$\AW  := \cup_{L>0}  \AW_{L^*}$,
\ie $\AW$ is the set of all \astar{} weights.
Note that any \astar{} weight is doubling, \ie
$\AW_{L^*}  \subset \DW_{L}$,
where $L$ depends only on $L^*$.
This  was proved in \cite{mt2000} and is an immediate consequence of the fact (see \cite[Theorem 6.1]{mt2000}) that  if
 $w\in\AW_{L^*}$ then,  for some
  $l$ depending only on $L^*$ (for example, $l=2L^*$ will do),   $w(I_1) \geq (|I_1|/|I_2|)^l w(I_2)$, for all intervals $I_1, I_2 \subset [-1,1]$ such that  $I_1\subset I_2$.
Indeed, for any $I\subset [-1,1]$, this implies
$w(I) \geq \left(  |I|/|2I\cap [-1,1]| \right)^l w(2I) \geq  2^{-l}   w (2I)$, which shows that $w \in \DW_{2^l}$.

Moreover, it is known and is not difficult to check (see   \cite[pp. 58 and 68]{mt2000}) that all $A^*$ weights are $A_\infty$ weights. Here, $A_\infty$ is the union of all Muckenhoupt $A_p$ weights
and can be defined as the set of all weights $w$ such that, for any $0<\alpha<1$, there is $0<\beta<1$ so that $w(E)\geq \beta w(I)$, for all intervals $I\subset [-1,1]$ and all measurable subsets $E\subset I$ with $|E| \geq \alpha |I|$ (see \eg \cite{stein}*{Chapter V}).

Clearly, any \astar{} weight on $[-1,1]$ is bounded since if $w\in\AW_{L^*}$, then $w(x) \leq L^* w[-1,1]/2$, $x\in [-1,1]$. (We  slightly abuse the notation and write $w[a,b]$ instead of $w \left([a,b]\right)$ throughout this paper.) At the same time,
 not every bounded doubling weight is an \astar{} weight (for example, the doubling weight constructed in \cite{fm} is bounded and is not in $A_\infty$, and so it is not an \astar{} weight either).

Throughout this paper, we use the standard notation
$\norm{f}{I} := \norm{f}{\L_\infty(I)} :=   \esssup_{u\in I} |f(u)|$ and  $\norm{f}{} := \norm{f}{[-1,1]}$.
 Also,
\[
E_n(f, I)_{w} := \inf_{q  \in\Poly_n} \norm{w(f-q )}{I} ,
\]
where $\Poly_n$ is the space of algebraic polynomials of degree $\leq n-1$.

The following theorem is due to G. Mastroianni and V. Totik \cite{mt2001}*{Theorem 1.4} and is the main motivation for the present paper (see also \cites{mt1998, mt1999, mt2000}).

\begin{theorema}[\cite{mt2001}*{Theorem 1.4}] \label{mtthm}
Let $r\in\N$, $M\geq 3$, $-1 = z_1 < \dots < z_M = 1$, and let $w$ be a bounded generalized Jacobi weight
\be \label{genJacobi}
\wj(x) := \prod_{j=1}^M |x-z_j|^{\gamma_j}   \quad \mbox{\rm with }\;  \gamma_j \geq 0, \; 1\leq j\leq M.
\ee
Then there is a constant $c$ depending only on $r$ and the weight $w$ such that, for any $f$,
\[
E_n(f, [-1,1])_{\wj} \leq c \w_\varphi^r (f, 1/n)_{\wj}^* ,
\]
and
\[
\w_\varphi^r (f, 1/n)_{\wj}^* \leq c n^{-r} \sum_{k=1}^n k^{r-1} E_k(f, [-1,1])_{\wj} ,
\]
where
\[
\w_\varphi^r (f, t)_{\wj}^* := \sum_{j=1}^{M-1} \sup_{0<h\leq t} \norm{\wj(\cdot) \Delta_{h\varphi(\cdot)}^r(f,\cdot, J_{j,h})}{}
+ \sum_{j=1}^M E_r(f, I_{j, t})_{\wj}
\]
with $I_{1,h} = [-1, -1+h^2]$, $I_{M, h} = [1-h^2, 1]$, $J_{1, h} = [-1+h^2, z_2-h]$, $J_{M-1, h} = [z_{M-1}+h, 1-h^2]$, and
$I_{j, h} = [z_j-h, z_j+h]$ and $J_{j, h} = [z_j+h, z_{j+1}-h]$ for $1<j<M-1$, and the $r$th  symmetric difference is defined in \ineq{dd}.
\end{theorema}

The purpose of the present paper is to prove an analog of \thm{mtthm} for more general weights (namely, for   \astar{} weights having finitely many zeros inside $[-1,1]$ and not  too ``rapidly changing'' away from these zeros), and give a more natural and transparent (in our opinion) definition of the modulus of smoothness $\w_\varphi^r$. Our recent paper \cite{k-singular} deals with approximation in the weighted $\L_p$, $p<\infty$, (quasi)norm and a certain class of doubling weights having finitely many zeros and singularities. Approximation in the weighted $\L_\infty$ norm considered in the current paper is similar in some sense, but it also presents some challenges that have to be dealt with, and our present proofs are
different from those in both \cite{mt2001} and \cite{k-singular}.
The main results of the present paper are \thm{jacksonthm} (direct result), \thm{conversethm} (inverse result) and \thm{corr99} (equivalence of the modulus and an appropriate realization functional).
Finally, we mention that \thm{mtthm} is a corollary of our results taking into account that $\wj \in \W^*(\Z)$, $\Z\in\bZM$ (see \rem{rem3.3}), and
\begin{eqnarray*}
\lefteqn{ \w_\varphi^r\left(f, \max\left\{(1-z_2^2)^{-1/2},(1-z_{M-1}^2)^{-1/2}\right\}, 1/2,  t\right)_{\wj} }\\
&\leq& \w_\varphi^r (f, t)_{\wj}^*  \leq  M\cdot \w_\varphi^r\left(f, 1/2, \max\left\{(1-z_2^2)^{-1/2},(1-z_{M-1}^2)^{-1/2}\right\}, t\right)_{\wj}, \quad 0<t\leq1,
\end{eqnarray*}
where $\W^*(\Z)$ and  $\w_\varphi^r(f, A, B, t)_w$ are defined in \deff{def11} and \ineq{compmod}, respectively.

\section{Some properties of \astar{} weights}

Note that, for any interval $I\subset [-1,1]$ and $x\in I$,
if   \ineq{astar} holds for $I_1 := I \cap [-1,x]$ and $I_2 := I \cap [x, 1]$, then it also holds for $I$ since $|I_1|+|I_2| = |I|$ and $w(I_1) + w(I_2) = w(I)$.
Therefore, $w\in\AW_{L^*}$   if and only if, for all intervals $[a,b] \subset [-1,1]$,
\be \label{astar1}
\max\{ w(a), w(b)\}  \leq {L^*  \over b-a}   w[a,b] .
\ee


\begin{lemma} \label{lemma02}
Let $w\in\AW_{L^*}$, $\xi\in [-1,1]$, and let   $w_1(x) := f(|x-\xi|)$, where $f: [0,2] \mapsto \R_+$ is nondecreasing and such that $f(2x)\leq \K f(x)$, for some $\K>0$ and all $0\leq x \leq 1$.  Then, $\widetilde w := w w_1 \in\AW_L$ with the constant $L$ depending only on $\K$ and $L^*$.
\end{lemma}

\begin{proof}
Suppose that   $I\subset [-1,1]$  and   $d$ is one of the endpoints of $I$. We need to show that $\widetilde w(d) \leq L \widetilde w(I)/|I|$.

{\bf Case 1: $\xi\not\in int(I)$.}\\
 Then, $w_1$ is monotone on $I$, and so either $w_1(d) \leq w_1(u)$ or $w_1(d) \geq w_1(u)$, for $u\in I$.
In the former case, we immediately have
\[
\widetilde w (d) = w(d) w_1(d) \leq  {L^*  \over |I|}  \int_I   w_1(d) w(u) du \leq {L^*  \over |I|} \widetilde w(I) .
\]
Suppose now that $w_1(d) \geq w_1(u)$, for $u\in I$. This means that $d$ is the endpoint of $I$ furthest from $\xi$.
Let $\zeta$ be the midpoint of $I$, and let $J := [d,\zeta]$ (as usual, if $x<y$, then $[y,x] := [x,y]$). Then,  $w_1(\zeta) \leq w_1(u)$, for all $u\in J$.
Also, since $|d-\xi|/2 \leq |\zeta-\xi|$ and $|d-\xi| \leq 2$,
we conclude that
\[
w_1(d) = f(|d-\xi|) \leq \K f(|d-\xi|/2) \leq  \K f(|\zeta-\xi|)= \K w_1(\zeta) .
\]
Therefore,
$w_1(d) \leq \K w_1(u)$, for all $u\in J$, and so
\begin{eqnarray} \label{auxw}
\widetilde w (d) &=&  w(d) w_1(d) \leq  {L^*  \over |J|}  \int_J   w_1(d) w(u) du \leq {L^* \K  \over |J|}  \int_J   w_1(u) w(u) du
 \leq  {L^*\K  \over |J|} \widetilde w(I) \\ \nonumber
 &=&  {2L^*\K  \over |I|} \widetilde w(I) .
\end{eqnarray}

{\bf Case 2: $\xi \in int(I)$.} \\
If $|d-\xi|\geq |I|/4$, then using  \ineq{auxw}  for  $I':= [d,\xi]$, we have
\[
\widetilde w (d) \leq {2L^*\K  \over |I'|} \widetilde w(I') \leq {8L^*\K  \over |I|}\widetilde w(I) .
\]
We now assume that $|d-\xi| < |I|/4$. Let $d'$ be the point symmetric to $d$ about $\xi$, \ie $\xi = (d+d')/2$, and let $I'':= I\setminus [d,d')$. Then
$|I''| = |I| - 2|d-\xi| \geq |I|/2$, and $w_1(d)=w_1(d') \leq w_1(u)$, for all $u\in I''$. Hence, taking into account that $w$ is doubling with the doubling constant depending only on $L^*$, we have
\[
\widetilde w (d) = w(d) w_1(d) \leq  {L^* w_1(d)  \over |I|}  \int_I    w(u) du \leq  {c w_1(d) \over |I|} \int_{I''}    w(u) du \leq
{c  \over |I|} \int_{I''}  w_1(u)  w(u) du \leq {c  \over |I|} \widetilde w (I) .
\]
This completes the proof.
\end{proof}

\begin{corollary} \label{corast}
Suppose that $w\in\AW_{L^*}$, $M\in\N$ and, for each $1\leq i\leq M$, $z_i \in [-1,1]$, $\gamma_i \geq 0$ and $\Gamma_i \in\R$ (if $\gamma_i >0$) or  $\Gamma_i \leq 0$ (if $\gamma_i =0$). Then
\be \label{genDT}
\widetilde w(x) := w(x) \prod_{i=1}^M |x-z_i|^{\gamma_i} \left( \ln{e \over |x-z_i|} \right)^{\Gamma_i}
\ee
 is an \astar{} weight with the \astar{} constant depending only on   $\gamma_i$'s,  $\Gamma_i$'s and $L^*$.
\end{corollary}

We remark that, with $w\sim 1$, the weights $\widetilde w$ in \ineq{genDT} are sometimes called ``generalized Ditzian-Totik weights''.

\begin{proof}  Denote
\[
f_{\gamma, \Gamma} (x) :=
\begin{cases}
\left( 1 - \ln x  \right)^{\Gamma} , & \mbox{\rm if  $\gamma  = 0$ and $\Gamma \leq 0$,} \\
x^\gamma  \left( \Psi - \ln  x  \right)^{\Gamma}  , & \mbox{\rm if  $\gamma  > 0$ and $\Gamma \in\R$,} \\
\end{cases}
\]
where $\Psi := 1 + \max\{0, \Gamma\}/\gamma$.
It is easy to check that $f_{\gamma, \Gamma}$ is nonnegative and nondecreasing on $[0,2]$, and
satisfies $\sup_{x\in [0,1]} |f_{\gamma, \Gamma}(2x)/f_{\gamma, \Gamma}(x)| < \infty$.
Hence, \lem{lemma02} implies that the weight
\[
\widehat w(x) := w(x) \prod_{i=1}^M  f_{\gamma_i, \Gamma_i} (|x-z_i|)
\]
 is an \astar{} weight with the \astar{} constant depending only on $\gamma_i$'s, $\Gamma_i$'s,  and $L^*$.

Finally, it remains to notice that, if $\gamma >0$ and $\Gamma \in\R$, then
$f_{\gamma, \Gamma}(x) \sim x^\gamma  \left( 1 - \ln  x  \right)^{\Gamma}$ on $[0,2]$ with equivalence constants depending only on $\gamma$ and $\Gamma$, and so $\widetilde w \sim \widehat w$ on $[-1,1]$. Clearly, this implies that $\widetilde w\in\AW$.
\end{proof}

\begin{remark} \label{remark25}
It follows from \cor{corast} that, for any \astar{} weight $w$ and any $\mu\geq 0$,
$w\varphi^\mu$ is also an \astar{} weight, where $\varphi(x) := \sqrt{1-x^2}$.
\end{remark}

For $n\in\N$,  following \eg \cite{mt2001}, we denote
\[
w_n(x) := \rho_n(x)^{-1} \int_{x-\rho_n(x)}^{x+\rho_n(x)} w(u) du ,
\]
 where $\rho_n(x) := n^{-1}\varphi(x) + n^{-2}$ (recall that $w$ is assumed to be $0$ outside $[-1,1]$). Note that, for any $w\in\AW_{L^*}$ and $x\in [-1,1]$,
\begin{eqnarray} \label{wlesswn}
w (x) &\leq&  {L^* \over \left| [x-\rho_n(x), x+\rho_n(x)] \cap [-1,1] \right|} \int_{[x-\rho_n(x), x+\rho_n(x)] \cap [-1,1]} w(u) du \\ \nonumber
& \leq &
{L^* \over  \rho_n(x) } \int_{x-\rho_n(x)}^{x+\rho_n(x)} w(u) du =  L^*  w_n(x) .
\end{eqnarray}

\begin{lemma}\label{lemma26}
Let $w\in\AW_{L^*}$  and   $n\in\N$. Then $w_n\in\AW_{L}$ with $L$ depending only on $L^*$.
\end{lemma}

\begin{proof} Suppose that   $n\in\N$ is fixed.
Let $I$ be a subinterval of $[-1,1]$,  and suppose that $x\in I$   is the left endpoint  of $I$ (the case for the right endpoint is analogous).
If $[x,x+\rho_n(x)] \subset I$, using the fact that $w$ is doubling, we have
\begin{eqnarray*}
w_n(x) &=& \rho_n(x)^{-1} \int_{x-\rho_n(x)}^{x+\rho_n(x)} w(u) du
 \leq
c \rho_n(x)^{-1} \int_{x}^{x+\rho_n(x)} w(u) du \\
&\leq &
c \rho_n(x)^{-1} \int_{x}^{x+\rho_n(x)}
{L^* \over \left| I \right|} \int_{I} w(v) dv \, du
 \leq
{c \over \left| I \right|} \int_{I} w(v) dv
  \leq   {c \over \left| I \right|}   \int_I w_n(v) dv .
\end{eqnarray*}
Recall now that, if $|x-u| \leq \K \rho_n(x)$, then $w_n(x) \sim w_n(u)$ (see \eg \cite[(2.3)]{mt2001}).
This implies that,
if $x$ is the left endpoint   of $I$   and $x+\rho_n(x) \not\in I$, then $I \subset [x, x+\rho_n(x)]$, and so $w_n(u) \sim w_n(x)$, for all $u\in I$. Hence, in this case,
\[
w_n(x) \sim {1 \over \left| I \right|}   \int_I w_n(u) du .
\]
Therefore,  \ineq{astar1} implies that
 $w_n$ is  an \astar{} weight.
\end{proof}


\sect{Special \astar{} weights    and associated moduli of smoothness}

Let
\[
\rho(h,x) := h\varphi(x)+h^2
\]
(note that $\rho(1/n, x)=\rho_n(x)$), and
\[
 \bZM := \left\{
(z_j)_{j=1}^M \st   -1 \leq    z_1 < \dots < z_{M-1} < z_M \leq  1 \right\}, \; M\in\N .
\]


For $\Z\in\bZM$, it is convenient to denote
 \[
 \Z_{A,h}^j  := \Z_{A,h}^j(\Z)  :=
 \left\{ x \in [-1,1] \st |x-z_j| \leq A \rho(h, z_j)  \right\},  \quad   1\leq j\leq M   ,
 \]
 \[
 \Z_{A,h}  :=  \Z_{A,h}(\Z)  :=
 \cup_{j=1}^M \Z_{A,h}^j,
 \]
 and
 \[
 \I_{A, h} :=   \I_{A, h}(\Z) :=
 \left([-1,1] \setminus  \Z_{A,h}\right)^{cl}   = \left\{ x\in [-1,1] \st |x-z_j| \geq  A \rho(h,z_j), \; \text{for all } 1\leq j \leq M \right\}.
 \]
Also,
 \[
  \D   := \pmin  \left\{  |z_j - z_{j-1}| \st 1\leq j\leq M+1 \right\} ,
 \]
 where $z_0 := -1$, $z_{M+1} := 1$
 and
 $\pmin(S)$ is the smallest {\em positive} number from the finite set $S$ of nonnegative reals.
 Note that $\D \leq 2$, for any $\Z\in\bZM$.

The following definition is an analog of \cite{k-singular}*{Definition 2.1} for \astar{} weights.

  \begin{definition} \label{def11}
  Let $\Z\in\bZM$.
 We say that   $w$ is an \astar{} weight from   the class
  $\W^*(\Z)$ (and write $w\in \W^*(\Z)$)   if
  \begin{itemize}
  \item[(i)]
   $w\in\AW$,
  \end{itemize}
    and
  \begin{itemize}
   \item[(ii)]
 for any $\e>0$ and  $x, y\in [-1,1]$ such that $|x-y| \leq  \rho(\e, x)$ and
 $\dist\left( [x,y], z_j \right) \geq \rho(\e, z_j)$ for all $1\leq j \leq M$, the following inequalities are satisfied
 \be \label{nochange}
c_* w(y) \leq  w(x) \leq c_*^{-1} w(y) , 
 \ee
 where the   constant $c_*$ depends only on $w$, and does not depend on $x$, $y$ and $\e$.
 \end{itemize}
 \end{definition}


Clearly, there are non-\astar{} weights  satisfying condition (ii) in Definition~\ref{def11}. For instance, the non-doubling  weight
\[
w(x) := \begin{cases}
-x , & \mbox{\rm if }\; x<0 ,\\
x^2 , & \mbox{\rm if }\; x\geq 0 ,
\end{cases}
\]
is one such example for $\Z := \{0\}$.

\begin{remark}
A weight from the class $\W^*(\Z)$  may  have zeros only at the points in $\Z$. At the same time, it is not required to have zeros at those points.
\end{remark}

\begin{remark} \label{rem3.3}
It follows from \cite{k-singular}*{Example 2.7} and \cor{corast}  that the following weights belong to $\W^*(\Z)$ with $\Z= (z_j)_{j=1}^M$, $-1\leq z_1 < \dots <z_{M-1} < z_M \leq 1$:
\begin{itemize}
\item
bounded classical Jacobi weights: $w(x) = (1+x)^\alpha (1-x)^\beta$, $\alpha,\beta \geq 0$,   with  $M=2$, $z_1 = -1$ and  $z_2= 1$,
\item
bounded generalized Jacobi weights \ineq{genJacobi},
\item
bounded generalized Ditzian-Totik weights \ineq{genDT} with $w\equiv 1$.
\end{itemize}
\end{remark}

The following lemma immediately follows from \cite{k-singular}*{Lemma 2.3} taking into account the fact that any \astar{} weight is doubling.

\begin{lemma} \label{properties}
Let $w$ be an \astar{} weight and $\Z\in\bZM$. The following conditions   are equivalent.
\begin{enumerate}[\rm (i)]
\item \label{i} $w\in \W^*(\Z)$.
\item \label{ii} For any $n\in\N$ and $x, y$ such that $[x,y] \subset \I_{1, 1/n}$ and $|x-y| \leq  \rho_n(x)$, inequalities \ineq{nochange} are satisfied with the constant $c_*$ depending only on $w$.

\item \label{iii}
For some $N\in\N$ that depends only on $w$, and any  $n\geq N$ and $x, y$ such that $[x,y] \subset \I_{1, 1/n}$ and $|x-y| \leq  \rho_n(x)$, inequalities \ineq{nochange} are satisfied with the constant $c_*$ depending only on  $w$.

\item  \label{iv} For any $n\in\N$, $A, B >0$, and $x, y$ such that $[x,y] \subset \I_{A, 1/n}$ and $|x-y| \leq B \rho_n(x)$, inequalities \ineq{nochange} are satisfied with the constant $c_*$ depending only on
$w$, $A$ and $B$.
\item \label{v}
For  any $n\in\N$ and $A>0$,
 \[
  w(x) \sim  w_n(x) , \quad x \in \I_{A, 1/n} ,
\]
where the equivalence constants depend only on $w$ and $A$, and are independent of $x$ and $n$.

\end{enumerate}
\end{lemma}

For $r\in\N$, $t>0$ and $\Z\in\bZM$, {\em the main part weighted modulus of smoothness} is defined as
 \be \label{mpmod}
 \Omega_\varphi^r(f, A, t)_{ w}  :=    \Omega_\varphi^r(f, A, t; \Z)_{ w}  :=
  \sup_{0<h\leq t} \norm{w(\cdot) \Delta_{h\varphi(\cdot)}^r(f,\cdot, \I_{A, h})}{}  ,
\ee
where
\be \label{dd}
\Delta_h^r(f,x, J):=\left\{
\begin{array}{ll} \ds
\sum_{i=0}^r  {r \choose i}
(-1)^{r-i} f(x-rh/2+ih),&\mbox{\rm if }\, [x-rh/2, x+rh/2]  \subset J \,,\\
0,&\mbox{\rm otherwise},
\end{array}\right.
\ee
is  the $r$th symmetric difference.

Note that if we denote
\be \label{domain}
 \Dom (A,h, r) :=   \left\{ x \st [x-rh\varphi(x)/2,x+rh\varphi(x)/2] \subset \I_{A, h} \right\}
\ee
then
\[
 \Omega_\varphi^r(f, A, t)_{ w}  =  \sup_{0<h\leq t} \norm{w(\cdot) \Delta_{h\varphi(\cdot)}^r(f,\cdot,\R)}{\Dom(A, h,r)} .
\]

The {\em weighted Ditzian-Totik   modulus of smoothness} is
\[
\w_\varphi^r(f, t)_{ w} :=  \sup_{0<h\leq t}  \norm{w(\cdot) \Delta_{h\varphi(\cdot)}^r(f,\cdot, [-1,1])}{} .
\]

For $A, B, t>0$, we define the  {\em complete weighted modulus of smoothness }   as
 \be \label{compmod}
 \w_\varphi^r(f, A, B, t)_{ w}  :=   \w_\varphi^r(f, A, B, t; \Z)_{ w}  :=
 \Omega_\varphi^r(f, A, t;\Z)_{ w}  + \sum_{j=1}^M  E_r(f, \Z_{B,t}^j)_{w} .
\ee

We will also need the following auxiliary quantity (``restricted main part modulus''):
 \be \label{restmod}
   \Omega_\varphi^r(f,  t)_{S, w}  :=   \sup_{0<h\leq t} \norm{w(\cdot) \Delta_{h\varphi(\cdot)}^r(f,\cdot, S)}{}   ,
\ee
where $S$ is some subset (a union of intervals) of $[-1,1]$ that does not depend on $h$.

\section{Properties of main part and complete weighted moduli}

\begin{proposition}  \label{propprop}
For any weight function $w$ and a set $\Z\in\bZM$, the moduli defined in \ineq{mpmod}, \ineq{compmod} and \ineq{restmod} have the following properties:
\begin{enumerate}[\rm (i)]

\item \label{pi}
$\ds \Omega_\varphi^r(f, A, t)_{w} = \Omega_\varphi^r(f, A, \sqrt{2/A})_{w}$  for any $t\geq \sqrt{2/A}$;

\item \label{pii}
$\ds \w_\varphi^r(f, A, B, t)_{ w} = \w_\varphi^r(f, A, B, t_0)_{ w} \geq M E_r(f, [-1,1])_{w}$ for any $t\geq t_0 := \max\{\sqrt{2/A}, \sqrt{2/B} \}$;

\item \label{piii}
$\ds  \Omega_\varphi^r(f, A, t_1)_{ w} \leq  \Omega_\varphi^r(f, A, t_2)_{ w}$ and    $\ds \w_\varphi^r(f, A, B, t_1)_{ w} \leq \w_\varphi^r(f, A, B, t_2)_{ w}$   if  $0<t_1 \leq t_2$;

\item \label{piv}
$\ds \Omega_\varphi^r(f, A_1, t)_{ w} \geq  \Omega_\varphi^r(f, A_2, t)_{ w}$
and $\ds \w_\varphi^r(f, A_1, B, t)_{ w} \geq \w_\varphi^r(f, A_2, B, t)_{ w}$
if $A_1 \leq A_2$;

\item \label{pv}
$\ds  \w_\varphi^r(f, A, B_1, t)_{ w} \leq \w_\varphi^r(f, A, B_2, t)_{ w}$
if $B_1 \leq B_2$;

\item \label{pvi}
$\ds \Omega_\varphi^r(f,  c_* t)_{\I_{A, t}, w} \leq  \Omega_\varphi^r(f, A/\max\{c_*, c_*^2\}, c_*t)_{ w}$ for any $t>0$ and $c_*>0$.

\end{enumerate}
\end{proposition}

\begin{proof}
Properties~(\ref{pi}) and (\ref{pii}) immediately follow from the observation that, if $h\geq \sqrt{2/C}$, then $C \rho(h, z_j) \geq 2$.
Properties~(\ref{piii}) and (\ref{pv})   follow from the definition and the fact that
$\Z^j_{B_1,  t_1} \subset \Z^j_{B_2,  t_2}$ if $t_1 \leq t_2$ and $B_1 \leq B_2$.
Property~(\ref{piv}) is a consequence of the inclusion $\I_{A_2, h} \subset \I_{A_1, h}$ if $A_1 \leq A_2$.
Property~(\ref{pvi}) follows from the observation that,
for  $c_*>0$ and  $0< h\leq c_*t$,  since  $\rho(h, z_j)/\max\{c_*, c_*^2\} \leq \rho(t, z_j)$, then
$\I_{A, t}  \subset \I_{A/\max\{c_*, c_*^2\}, h}$.
\end{proof}

We  need an auxiliary lemma that is   used in the proofs of several results below.

\begin{lemma}\label{auaulemma}
Suppose that  $\Z\in\bZM$ and $w \in \W^*(\Z)$.
If $A, h>0$, $r\in\N$ and $x\in [-1,1]$ are such that
\[ 
[x-rh\varphi(x)/2,  x+rh\varphi(x)/2 ] \subset \I_{A, h} \quad \mbox{\rm (\ie $x\in  \Dom (A,h, r)$),}
\]
then, for any $y\in [x-rh\varphi(x)/2,  x+rh\varphi(x)/2 ]$,
\[
w(y) \sim w(x) \sim w_{n}(x) ,
\]
where $n := \lceil 1/h \rceil$, and the equivalence constants depend only on $r$, $A$ and the weight $w$.
\end{lemma}

\begin{proof} First we note that, if $h > \sqrt{2/A}$, then $A\rho(h, z_j) > 2$, and so $\I_{A, h}=\emptyset$. Hence, we can assume that $0<h\leq \sqrt{2/A}$.
Now, if $n  = \lceil 1/h \rceil$, then  $n\in\N$,  $n^{-1} \leq h < (n-1)^{-1}$ and $\I_{A, h} \subset \I_{A, 1/n}$.
Moreover, if $n\geq 2$, then $(n-1)^{-1} \leq 2/n$ and so
$\rho(h, x) \leq 4 \rho_n(x) $ and, if $n=1$, then
 \[
\rho(h, x) \leq  \rho(\sqrt{2/A}, x) \leq \max\{ \sqrt{2/A}, 2/A \} \rho_n(x) .
 \]
Hence, if $y\in [x-rh\varphi(x)/2,  x+rh\varphi(x)/2 ]$, then $[x,y] \subset \I_{A, 1/n}$ and
\[
|x-y| \leq rh\varphi(x)/2
\leq
r\rho(h,x)/2 \leq (r/2) \max\{4, \sqrt{2/A}, 2/A \}  \rho_n(x) .
\]
Therefore, \lemp{iv} implies that $w(y)\sim w(x)$, and \lemp{v} yields the equivalence $w(x)\sim w_n(x)$.
\end{proof}

In the following lemma and in the sequel, we
  use the usual notation
\[
\L_\infty^w := \left\{ f: [-1,1]\mapsto\R \st \norm{wf}{} <\infty \right\} .
\]

\begin{lemma} \label{normestimate}
If $\Z\in\bZM$, $w \in \W^*(\Z)$,   $f\in\L_\infty^w$,
  $r\in\N$, and $A, B,  t >0$,  then
\[
 \w_\varphi^r(f, A, B,  t)_{ w}    \leq     c \norm{wf}{}  ,
\]
where $c$ depends only on  $r$, $A$  and the weight $w$.
\end{lemma}

\begin{proof}
First of all, it is clear that
\[
\sum_{j=1}^M  E_r(f, \Z_{B,t}^j)_{w}
\leq \sum_{j=1}^M  \norm{wf}{\Z_{B,t}^j} \leq M \norm{wf}{}.
\]
We now let $h \in (0, t]$   and $x$ be such that $[x-rh\varphi(x)/2,  x+rh\varphi(x)/2 ] \subset \I_{A, h}$, and denote  $y_i(x) := x+(i-r/2)h \varphi(x)$.
Then,
\lem{auaulemma} implies that $w(y_i(x)) \sim w(x)$, $0\leq i \leq r$,  and so

 \begin{eqnarray*}
    w(x) \left|\Delta_{h\varphi(x)}^r(f,x, \I_{A,h})\right|   & \leq &     w(x) \sum_{i=0}^r  {r \choose i}
 \left| f(y_i(x)) \right|
 \leq   2^r w(x) \max_{0 \leq i \leq r} \left| f(y_i(x)) \right| \\
 & \leq &  c    \max_{0 \leq i \leq r} \left|w(y_i(x))  f(y_i(x)) \right| .
\end{eqnarray*}

This   yields
$ \Omega_\varphi^r(f, A, t)_{ w}\leq     c \norm{wf}{}$, which completes the proof of the lemma.
\end{proof}

Taking into account that $\w_\varphi^r(f, A, B,  t)_{ w} = \w_\varphi^r(f-q, A, B,  t)_{ w}$, for any $q\in\Poly_r$, we immediately get the following corollary.

\begin{corollary} \label{cornormestimate}
If $\Z\in\bZM$, $w \in \W^*(\Z)$,   $f\in\L_\infty^w$,
  $r\in\N$, and $A, B,  t >0$,  then
\[
 \w_\varphi^r(f, A, B,  t)_{ w}    \leq   c E_r(f, [-1,1])_{w},
\]
where $c$ depends only on  $r$, $A$  and the weight $w$.
\end{corollary}

\begin{lemma} \label{twot}
If $\Z\in\bZM$, $w \in \W^*(\Z)$,   $f\in\L_\infty^w$,
  $r \in\N$,  and $A, t >0$,  then
\be \label{ineqtwot}
\Omega_\varphi^r(f, A, 2 t)_{ w}    \leq      c   \Omega_\varphi^r(f, \sqrt{2} A , \sqrt{2} t)_{ w} ,
\ee
where $c$ depends only on  $r$, $A$  and the weight $w$.
\end{lemma}

Now, Proposition~\ref{propprop}(\ref{piii} and \ref{piv}) and \lem{twot} imply the following result.

\begin{corollary} \label{corollary411}
If $\Z\in\bZM$, $w \in \W^*(\Z)$,   $f\in\L_\infty^w$,
  $r \in\N$,  and $A, t >0$,  then
\[
\Omega_\varphi^r(f, A, t)_{ w} \sim \Omega_\varphi^r(f, \sqrt{2} A, t)_{ w} ,
\]
 and so
\[
\Omega_\varphi^r(f, A, t)_{ w} \sim \Omega_\varphi^r(f, 1 , t)_{ w} ,
\]
where the equivalence constants depend  only on  $r$, $A$  and the weight $w$.

Moreover,
\[
\Omega_\varphi^r(f, 1 , t)_{ w} \leq \Omega_\varphi^r(f, 1 , 2t)_{ w} \leq c \Omega_\varphi^r(f, 1 , t)_{ w} ,
\]
where $c$ depends only on $r$ and the weight $w$.
\end{corollary}

\begin{proof}[Proof of \lem{twot}]
Recall a rather well known identity (see \cite{pp}*{(5) on p. 42}, for example)
\be \label{pp}
\Delta_{2 h}^r (f, x) = \sum_{i_1=0}^{1} \dots \sum_{i_r=0}^{1} \Delta_{h}^r \left(f, x + [i_1 +\dots +i_r - r/2]h \right)  .
\ee

Now, we fix $h \in (0, t]$, and let   $x$  be  a fixed number  such that $[x-r h\varphi(x) ,  x+r h\varphi(x)  ] \subset \I_{A, 2h}$ (\ie $x\in\Dom(A, 2h, r)$).
We have
 \begin{eqnarray*}
     \left|\Delta_{2h\varphi(x)}^r(f,x, \I_{A, 2h})\right|
& \leq &   \sum_{i_1=0}^{1} \dots \sum_{i_r=0}^{1} \left|\Delta_{h\varphi(x)}^r \left(f, x + [i_1 +\dots +i_r - r/2]h\varphi(x) \right)\right| \\
& \leq &
 2^r  \left|\Delta_{h\varphi(x)}^r \left(f, y \right)\right|  =:  2^r  F ,
 \end{eqnarray*}
where $y  := x + \gamma h\varphi(x)$, and
$\gamma$ is such that $\gamma +r/2 \in \{0, 1, \dots, r\}$ (and so $|\gamma| \leq r /2$) and
\[
\left|\Delta_{h\varphi(x)}^r \left(f, y \right)\right|  = \max_{0\leq m \leq r}
 \left|\Delta_{h\varphi(x)}^r \left(f, x + [m - r/2]h\varphi(x) \right)\right| .
\]
Note that \lem{auaulemma} implies that $w(x)\sim w(y)$.
Also,
since $x\pm r h\varphi(x) \in [-1,1]$, we have $|x| \leq (1-r^2h^2)/(1+r^2h^2)$, which implies
\[
{|x| \over \varphi(x)} \leq {1-r^2h^2 \over 2rh} ,
\]
and so
\begin{eqnarray*}
\left[\varphi(y ) \over \varphi(x) \right]^2 & = & 1 - \gamma^2 h^2 - 2\gamma h {x \over \varphi(x)}
\geq 1 - \gamma^2 h^2 - 2 |\gamma| h {|x| \over \varphi(x)}
 \geq \frac{1}{2} + \frac{r^2h^2}{4}  \geq \frac{1}{2}.
\end{eqnarray*}
Therefore, $ \varphi(x) \leq \sqrt{2} \varphi(y )$, and
\[
w(x) F  \leq c w (y) \left|\Delta_{h^*\varphi(y)}^r \left(f,  y  \right)\right| ,
\]
where
\[
0<h^* :=   {h\varphi(x) \over \varphi(y)} \leq \sqrt{2} h \leq \sqrt{2} t .
\]
We now note that $\rho(2h,z_j) \geq \sqrt{2}\rho(h^* , z_j)$ which implies $\I_{A, 2h} \subset \I_{\sqrt{2}A , h^*}$, and so
 \begin{eqnarray*}
&&  \left[y  - rh^*\varphi (y)/2, y  + rh^*\varphi (y)/2 \right]
  =
\left[ x +(\gamma-r/2) h \varphi(x), x +(\gamma+r/2) h \varphi(x) \right]  \\
&& \mbox{} \subset   \left[ x -r h \varphi(x), x + r h \varphi(x) \right] \subset \I_{A, 2h} \subset \I_{\sqrt{2}A , h^*}.
 \end{eqnarray*}
Therefore, $\Delta_{h^*\varphi(y)}^r \left(f,  y  \right) = \Delta_{  h^* \varphi(y)}^r \left(f,  y , \I_{\sqrt{2}A,  h^*} \right)$, and so
we have
\begin{eqnarray*}
w(x) F &\leq &
 c  \sup_{0<h^* \leq \sqrt{2} t}
 \esssup_{y }
  w (y) \left|\Delta_{  h^* \varphi(y)}^r \left(f,  y , \I_{\sqrt{2}A,  h^*} \right)\right|
\leq  c \Omega_\varphi^r(f, \sqrt{2}A , \sqrt{2} t)_{ w}
 \end{eqnarray*}
for almost all $x\in\Dom(A, 2h, r)$.
The lemma is now proved.
%
\end{proof}

\begin{lemma} \label{newlem21}
Let $\Z\in\bZM$,  $w \in \W^*(\Z)$, $r\in\N$, $z  \in \Z$, $z\neq 1$, $0<\e <\D/2$,   $I := [z+\e/2,z+\e]$, and
let $J := [z+\e,z+\e+\delta]$ with $\delta$ such that $0< \delta \leq \e/(2r)$.  Then, for any  $h\in [\delta, \e/(2r)]$
 and any polynomial $q\in\Poly_r$, we have
\be \label{osn}
\norm{w(f-q)}{J}
\leq
c \norm{ w(\cdot) \Delta_{h}^r (f, \cdot, I\cup J)}{I\cup J} + c  \norm{w(f-q)}{I}  .
\ee
Additionally,
\be \label{osn2}
\norm{w(f-q)}{J}
\leq
c  \Omega_\varphi^r(f, 54 \delta t/\e)_{I\cup J, w}
 + c  \norm{w(f-q)}{I} ,
\ee
where $0<t < 1$ is such that $\e = \rho(t,z)$, and all
constants $c$ depend only on $r$ and the weight $w$.
\end{lemma}

\begin{remark}
By symmetry, the statement of the lemma is also valid for $I:= [z-\e,z-\e/2]$ and $J := [z-\e-\delta,z-\e]$, where $z\in\Z$ is such that $z\neq -1$.
\end{remark}

\begin{remark}
The condition $\e <\D/2$ guarantees that $I$ is ``far'' from all other points in $\Z$. In particular,
 $[z+\e/2,z+2\e] \cap \left(\Z\cup\{\pm 1\}\right) = \emptyset$.
\end{remark}

\begin{proof}[Proof of \lem{newlem21}]
Denoting for convenience  $g := f-q$ and taking into account that  $\Delta_{h}^r (g, x, \R) = \Delta_{h}^r (f, x, \R)$  we have
 \[
 g(x+rh/2) = \Delta_{h}^r (f, x,\R) - \sum_{i=0}^{r-1}  {r \choose i} (-1)^{r-i} g(x-rh/2+ih).
 \]
We now fix  $h \in [\delta ,  \e/(2r)]$, and   note that, for any  $x$   such that $x+rh/2 \in J$, we have
 $[x-rh/2, x+(r-2)h/2] \subset I$, and so
\begin{eqnarray*}
\norm{g}{J} &\leq&  \norm{ \Delta_{h}^r (f, \cdot,\R)}{[z+\e-rh/2, z+\e+\delta - rh/2]} + (2^r-1) \norm{g}{I} \\
& \leq & \norm{ \Delta_{h}^r (f, \cdot, I\cup J)}{I\cup J} +  (2^r-1)  \norm{g}{I} .
\end{eqnarray*}
Suppose now that $0<t < 1$ is such that $\e = \rho(t, z)$, let $n\in\N$ be such that $n:= \lfloor 1/t \rfloor$ and pick $A$ so that $\e =A \rho_n(z)$.
Note that $\rho_{n+1}(z)  < \e \leq \rho_n(z)$ and  $1/4 < A \leq 1$. Hence,   $\dist\left(I\cup J, z\right) = \e/2 \geq \rho_n(z)/8$.
Suppose now that $\widetilde z \in \Z$ is such that $\widetilde z > z$ and $(z, \widetilde z) \cap \Z = \emptyset$, \ie $\widetilde z$ is the ``next'' point from $\Z$ to the right of $z$ (if there is no such $\widetilde z$ then there is nothing to do, and the next paragraph can be skipped).

We will now show that $\widetilde d := \dist\left(I\cup J, \widetilde z\right)  \geq   \rho_n(\widetilde z)/20$.
Indeed, $\widetilde d = \widetilde z - z - (\e+\delta) \geq \D -3\e/2 \geq \e/2$. If $\e > \rho_n(\widetilde z)/10$, then we are done, and so we suppose that
$\e \leq \rho_n(\widetilde z)/10$.
Recall (see \eg \cite{k-singular}*{p. 27}) the well known fact that
\be \label{usual}
\rho_n(u)^2 \leq 4 \rho_n(v) (|u-v|+\rho_n(v)) , \quad \mbox{\rm for all }\; u,v\in [-1,1] .
\ee
 This implies
\[
|\widetilde z - z|   \geq {\rho_n(\widetilde z)^2 \over 4 \rho_n(z)} - \rho_n(z)\geq {5A \over 2}  \rho_n(\widetilde z)  - {\e \over A} \geq \left( {5A \over 2} -  {1 \over 10 A}    \right) \rho_n(\widetilde z) \geq (9/40) \rho_n(\widetilde z) .
\]
Also, $|\widetilde z - z| = \widetilde d + \e+\delta \leq \widetilde d + 3\e/2 \leq 4 \widetilde d$, which implies $\widetilde d \geq (9/160) \rho_n(\widetilde z) \geq  \rho_n(\widetilde z)/20$ as needed.
Therefore, we can conclude that
\[
I\cup J \subset \I_{1/20, 1/n} .
\]
Now, using \ineq{usual} we conclude that, if $u\in I\cup J$, then $|u-z| \leq  3\e/2 \leq 3\rho_n(z)/2$, and so
 \[
 \rho_n(u)^2  \leq 4 \rho_n(z) (|u-z|+\rho_n(z)) \leq 10 \rho_n(z)^2
 \]
and
\[
\rho_n(z)^2 \leq 4 \rho_n(u) (|u-z|+\rho_n(u)) \leq 4 \rho_n(u) (3\rho_n(z)/2 +\rho_n(u)) .
\]
This  implies that, for any  $u\in I\cup J$,
\[
\rho_n(u)/4  \leq   \rho_n(z)  \leq 7 \rho_n(u) .
\]
Hence, for any $u, v \in I\cup J$,
\[
|u-v| \leq \e \leq \rho_n(z) \leq 7 \rho_n(u) .
\]
It now follows from \lemp{iv} that $w(u) \sim w(v)$, for any $u, v \in I\cup J$, and so
\[
\norm{w g}{J} \leq c \norm{ w(\cdot) \Delta_{h}^r (f, \cdot, I\cup J)}{I\cup J} + c  \norm{wg}{I} ,
\]
and \ineq{osn} is proved.

In order to prove  \ineq{osn2},
  we note that, for any $x\in I\cup J$,
\[
1-|x| \geq \e/2 = \rho (t,z)/2   \geq t^2/2  ,
\]
%
 %
 %
 which implies $\varphi(x)  \geq t/\sqrt{2}$,
and so, with $h:= \delta$, we have
\[
0< {h \over \varphi(x)}   \leq {\delta \rho_n(z) \over   \e \varphi(x)} \leq {7 \delta  \rho_n(x) \over  \e \varphi(x)}
\leq {7\delta \over \e n} \left(1 + {\sqrt{2} \over nt} \right) \leq
{7\delta t \over \e} \left(2 + 4\sqrt{2}  \right) \leq {54\delta t/\e} .
\]
Therefore, for almost all $x \in I\cup J$,   denoting  $h^* := h/\varphi(x)$ we have
\begin{eqnarray*}
w(x) \Delta_{h}^r (f, x, I\cup J) & = & w(x) \Delta_{h^* \varphi(x)}^r (f, x, I\cup J)    \\
&\leq & \sup_{0<h\leq 54\delta t/\e}  \norm{ w(\cdot) \Delta_{h \varphi(\cdot) }^r (f, \cdot, I\cup J)}{} ,
\end{eqnarray*}
and the proof of \ineq{osn2} is complete.
\end{proof}

\begin{corollary} \label{corol54}
Let $\Z\in\bZM$,  $w \in\W^*(\Z)$, $r\in\N$,   $B>0$, and let   $0<t < c_0$, 
where $c_0$ is such  that $\max_{1\leq j \leq M} \rho(c_0,z_j) \leq \D/(2B)$ (for example, $c_0 := \min\{ 1, \D/(4B)\}$ will do).
Then,
\[
\w_\varphi^r\left(f, 1,  B(1+1/(2r)) , t \right)_{ w} \leq c \w_\varphi^r(f, 1,  B , t)_{ w} ,
\]
where   the constant $c$ depends only on $r$,  $B$ and the weight $w$.
\end{corollary}

Taking into account that $(1+1/(2r))^m  \geq 2$ for $m = \lceil 1/\log_2(1+1/(2r))\rceil$, we
  immediately get the following result.

\begin{corollary} \label{cor2.14}
Let $\Z\in\bZM$,  $w \in\W^*(\Z)$, $r\in\N$,   $B>0$, and let   $0<t < c_0$,
where    $c_0$ is such   that $\max_{1\leq j \leq M} \rho(c_0,z_j) \leq \D/(2B)$ (for example, $c_0 := \min\{ 1, \D/(4B)\}$ will do).
Then,
\[
\w_\varphi^r\left(f, 1,  B , t \right)_{ w} \leq c \w_\varphi^r(f, 1,  B/2 , t)_{w} ,
\]
where   the constant $c$ depends only on $r$,  $B$ and the weight $w$.
\end{corollary}


\begin{proof}[Proof of \cor{corol54}]
%
%
For each $1\leq j \leq M$, let $\e_j :=   B \rho(t,z_j)$ and note that $\e_j < \D/2$. It follows from \lem{newlem21} and the remark after it   that, for any $q_j \in \Poly_r$,
 $ \delta_j :=  \e_j/(2r) $ and $\tau_j$ such that $\rho(\tau_j, z_j) = B \rho(t,z_j) = \e_j$, we have
 \[
\norm{w(f-q_j)}{J_j^r}
 \leq   c \Omega_\varphi^r(f, 27  \tau_j/r)_{I_j^r\cup J_j^r, w}  + c \norm{w(f-q_j)}{I_j^r} ,
\]
where
$I_j^r := [z_j+\e_j/2,z_j+\e_j]$ and  $J_j^r := [z_j+\e_j,z_j+\e_j+\delta_j]$, and
\[
\norm{w(f-q_j)}{J_j^l}
 \leq   c \Omega_\varphi^r(f, 27  \tau_j/r)_{I_j^l\cup J_j^l, w} + c \norm{w(f-q_j)}{I_j^l} .
\]
where
$I_j^l := [z_j-\e_j, z_j-\e_j/2]$ and  $J_j^l := [z_j-\e_j-\delta_j, z_j-\e_j]$.
Note that, if $z_j = 1$ or $-1$, we do not consider $I_j^r$, $J_j^r$ or $I_j^l$, $J_j^l$, respectively.

We now note that $[z_j-\e_j, z_j+\e_j] \cap [-1,1] = \Z_{B, t}^j$ and
$[z_j-\e_j-\delta_j, z_j+\e_j+\delta_j] \cap [-1,1] = \Z_{\widetilde B, t}^j$, where $\widetilde B := B(1 +1/(2r))$.
Letting $q_j\in \Poly_r$ be such that
 \[
 \norm{w(f-q_j)}{\Z_{B,t}^j} \leq c  E_r(f, \Z_{B,t}^j)_{w} ,
 \]
we have
 \begin{eqnarray*}
E_r(f, \Z_{\widetilde B, t}^j)_{w} & \leq &    \norm{w(f-q_j)}{\Z_{\widetilde B, t}^j} \\
&\leq &
  \norm{w(f-q_j)}{J_j^l} +   \norm{w(f-q_j)}{\Z_{ B, t}^j} +   \norm{w(f-q_j)}{J_j^r} \\
& \leq &
 c \norm{w(f-q_j)}{\Z_{ B, t}^j} +  c \Omega_\varphi^r(f, 27 \tau_j/r)_{I_j^l\cup J_j^l, w}
 + c   \Omega_\varphi^r(f, 27 \tau_j/r)_{I_j^r\cup J_j^r, w} \\
 & \leq &
 c E_r(f, \Z_{B,t}^j)_{ w} + c \Omega_\varphi^r(f, 27  \tau_j/r)_{I_j^l\cup J_j^l, w}
 + c   \Omega_\varphi^r(f, 27 \tau_j/r)_{I_j^r\cup J_j^r, w}  .
 \end{eqnarray*}
Now, note that
\[
I_j^r\cup J_j^r \subset \I_{B/2,t} \andd I_j^l\cup J_j^l \subset \I_{B/2,t} .
\]

Hence, taking into account that $\tau_j \leq \max\{B, \sqrt{B} \} t$ we get
 \begin{eqnarray*}
E_r(f, \Z_{\widetilde B, t}^j)_{w} &\leq&  c E_r(f,\Z_{B,t}^j)_{ w} +
c \Omega_\varphi^r(f, 27  \tau_j/r)_{\I_{B/2,t}, w}  \\
&\leq& c E_r(f,\Z_{B,t}^j)_{w} +
c \Omega_\varphi^r \left(f, 27 \max\{B, \sqrt{B} \} t/r\right)_{\I_{B/2,t}, w} .
 \end{eqnarray*}
Now, with $c_* :=  27 \max\{B, \sqrt{B} \}/r$,
Proposition~\ref{propprop}(\ref{pvi}) and \cor{corollary411} imply
\[
\Omega_\varphi^r \left(f, c_*  t\right)_{\I_{B/2,t}, w} \leq
\Omega_\varphi^r  \left( f, B/(2\max\{c_*, c_*^2\}), c_* t\right)_{ w}
\leq c \Omega_\varphi^r  \left( f, 1, t\right)_{w} .
\]
Therefore,
 \begin{eqnarray*}
 \w_\varphi^r(f, 1, \widetilde B, t)_{ w}  &=&   \Omega_\varphi^r(f, 1, t)_{ w}  + \sum_{j=1}^M  E_r(f, \Z_{\widetilde B,t}^j)_{ w  }  \\
 & \leq & c \Omega_\varphi^r(f, 1, t)_{ w} +  c \sum_{j=1}^M  E_r(f, \Z_{B,t}^j)_{w}  \\
 & \leq & c  \w_\varphi^r(f, 1,  B, t)_{w} ,
 \end{eqnarray*}
and the proof is complete.
\end{proof}

\sect{Auxiliary results}

\begin{theorem}[\mbox{\cite[(6.10)]{mt2000}}] \label{thmremez}
Let $W$ be a $2\pi$-periodic function which is an \astar{} weight on $[0, 2\pi]$. Then there is a constant $C>0$  such that if $T_n$ is a trigonometric polynomial of degree at most $n$ and $E$ is a measurable subset of $[0, 2\pi]$ of measure at most $\Lambda/n$, $1\leq \Lambda\leq n$, then
\[
\norm{ T_n W}{[0, 2\pi]}  \leq C^{\Lambda}   \norm{ T_n W }{[0, 2\pi]\setminus E} .
\]
\end{theorem}


The following result is essentially proved in \cite{mt2000}. However, since it was not stated there explicitly we sketch its very short proof below.

\begin{corollary} \label{cordw}
 Let $w\in\AW_{L^*}$.
If $E \subset [-1,1]$ is such that $ \int_E (1-x^2)^{-1/2} dx   \leq \lambda/n$ with $\lambda\leq n/2$, 
then for each $P_n\in\Pn$, we have
\[
\norm{P_n w}{[-1,1]} \leq c \norm{P_n w}{[-1,1]\setminus E} ,
\]
where the constant $c$ depends only on $\lambda$  and $L^*$.
\end{corollary}

\begin{proof}
Let $W(t) := w(\cos t)$,  $T_n(t) := P_n(\cos t)$ and $\widetilde E:= \left\{ 0\leq t \leq 2\pi \st \cos t \in E\right\}$. Note that   $W$ is a $2\pi$-periodic function which is an \astar{} weight on $[0, 2\pi]$ (see \cite[p. 68]{mt2000}), and
\[
\meas(\widetilde E) = \int_{\widetilde E} dt = 2 \int_{\widetilde E \cap [0, \pi]} dt = 2 \int_E  (1-x^2)^{-1/2} dx   \leq 2\lambda/n   .
\]
Hence,
\begin{eqnarray*}
\norm{P_n w}{} & = & \norm{T_n W}{[0,2\pi]} \leq c \norm{ T_n W }{[0, 2\pi]\setminus \widetilde E}
= c \norm{P_n w}{[-1,1]\setminus E}.
\end{eqnarray*}
\end{proof}

\begin{lemma}[\mbox{\cite[(7.27)]{mt2000}}] \label{auxlemma}
 Let $w$ be an \astar{} weight on $[-1,1]$.  Then, for all $n\in\N$ and  $P_n\in\Poly_n$,
\[
 \norm{P_n w}{} \sim \norm{P_n w_n}{}
\]
with the equivalence constants independent   of $P_n$ and $n$.
\end{lemma}

It is convenient to denote $\varphi_n(x) := \varphi(x) + 1/n$, $n\in\N$, and note that   $w := \varphi$ is an \astar{} weight and $w_n \sim \varphi_n$ on $[-1,1]$.

One of the applications of \cor{cordw} is the following quite useful result.

\begin{theorem} \label{mainauxthm}
 Let $w$ be a \astar{} weight,  $n\in\N$, $0\leq \mu \leq n$. Then, for any  $P_n \in \Poly_n$,
\be \label{mainauxineq1}
 \norm{w \varphi^\mu   P_n}{} \sim \norm{w_n \varphi^\mu  P_n}{}
\ee
and
\be \label{mainauxineq2}
\norm{w \lambda_n^\mu  P_n}{}  \sim \norm{w_n \lambda_n^\mu  P_n}{}    ,
\ee
where $\lambda_n(x) := \max\left\{  \sqrt{1-x^2} , 1/n \right\}$,
and
the equivalence constants are independent of $\mu$, $n$ and $P_n$.
\end{theorem}

 \begin{proof}
 We start with the equivalence \ineq{mainauxineq1}.
Let $m:= 2\lfloor \mu/2 \rfloor$. Then $m$ is an even integer such that $\mu-2 < m \leq \mu$ (note that $m=0$ if $\mu<2$), and
$Q_{n+m} := \varphi^m P_n \in \Poly_{n+m} \subset \Poly_{2n}$.

Since $w$ is an  \astar{} weight, then $w \varphi^{\gamma}$, $\gamma>0$,  is also an  \astar{} weight (see \rem{remark25}) and
 \[
 (w \varphi^{\gamma  })_n \sim w_n \varphi_n^{\gamma  },
 \]
where
the equivalence constants depend on $\lceil \gamma\rceil $  and the doubling constant of $w$. 

Hence, denoting $E_n :=[-1+n^{-2}, 1-n^{-2}]$,  $\eta := \mu-m$, noting that $0\leq \eta <2 $ (and so $\lceil \eta \rceil$ is either $0$, $1$ or $2$ allowing us to replace constants that depend on $\lceil \eta \rceil$ by those independent of $\eta$), and using Lemmas~\ref{auxlemma} and \ref{lemma26},  \cor{cordw}, and the observation that $w_n(x)\sim w_k(x)$ if $n\sim k$,   we have
\begin{eqnarray*}
\norm{\varphi^\mu w P_n}{} &=& \norm{\varphi^{\eta} w  Q_{n+m}}{}   \sim \norm{(w \varphi^{\eta })_n  Q_{n+m}}{}
\sim \norm{ (w \varphi^{\eta})_n Q_{n+m}}{E_n}  \\
& \sim & \norm{ w_n \varphi_n^{\eta } Q_{n+m}}{E_n}
\sim \norm{w_n \varphi^{\eta  }   Q_{n+m}}{E_n} .
\end{eqnarray*}
Since   $w_n \varphi^{\eta }$ is an \astar{} weight (see \rem{remark25}), we can continue as follows:
\begin{eqnarray*}
 \norm{ w_n \varphi^{\eta } Q_{n+m}}{E_n} \sim  \norm{ w_n \varphi^{\eta } Q_{n+m}}{}
 = \norm{w_n   \varphi^\mu   P_{n}}{} .
\end{eqnarray*}
Note that none of the constants in the equivalences above depend on $\mu$. This completes the proof of \ineq{mainauxineq1}.

\medskip

Now, let $\E_n := \left\{ x \st \sqrt{1-x^2} \leq 1/n \right\}$ and note that $\lambda_n(x) =1/n$ if $x\in \E_n$, and $\lambda_n(x) =\varphi(x)$ if $x\in [-1,1]\setminus \E_n$.
Using \ineq{mainauxineq1} we have
\begin{eqnarray*}
 \norm{w \lambda_n^\mu P_n}{}  & \leq &
 \norm{w \lambda_n^\mu P_n}{\E_n} + \norm{w \lambda_n^\mu P_n}{[-1,1]\setminus \E_n} \\
 & = &
 n^{- \mu}  \norm{w P_n}{\E_n} +   \norm{w \varphi^\mu P_n}{[-1,1]\setminus \E_n} \\
 & \leq &
 n^{- \mu}  \norm{w P_n}{} +   \norm{w \varphi^\mu P_n}{} \\
  & \leq &
c_0   \left( n^{- \mu}   \norm{w_n P_n}{} +   \norm{w_n \varphi^\mu P_n}{}\right) \\
 & \leq &
2c_0  \norm{w_n \lambda_n^\mu P_n}{} .
\end{eqnarray*}
In the other direction, the sequence of inequalities is exactly the same (switching $w$ and $w_n$).
This verifies \ineq{mainauxineq2}.
\end{proof}

If we allow constants to depend on $\mu$, then we have the following result.

\begin{corollary} \label{newjust1}
Let $w$ be an \astar{} weight,    $n\in\N$ and $\mu \geq 0$. Then, for any  $P_n \in \Poly_n$,
\[
\norm{w \varphi_n^\mu P_n}{} \sim    \norm{w \varphi^\mu P_n}{} \sim \norm{w_n \varphi^\mu P_n}{}   \sim \norm{ w_n \varphi_n^\mu P_n}{} ,
\]
where all equivalence constants 
are independent of $n$ and $P_n$.
\end{corollary}

\begin{proof}
Since $\lambda_n(x) \leq \varphi_n(x) \leq 2 \lambda_n(x)$ and $\varphi(x) \leq \varphi_n(x)$, we immediately get from \thm{mainauxthm}
\[
\norm{w \varphi^\mu P_n}{} \sim \norm{w_n \varphi^\mu P_n}{} \leq \norm{w_n\varphi_n^\mu P_n}{}\sim \norm{w\varphi_n^\mu P_n}{} .
\]
At the same time,
\[
\norm{w \varphi^\mu P_n}{}   \sim      \norm{ (w\varphi^{\mu })_n P_n}{}  \sim   \norm{w_n  \varphi_n^{\mu} P_n}{}.
\]
and the proof is complete.
\end{proof}

\begin{theorem}[Markov-Bernstein  type theorem] \label{thm5.5}
Let $w$ be an \astar{} weight and  $r\in\N$.  Then, for all $n\in\N$ and $P_n\in\Poly_n$,
\[
  n^{-r} \norm{w \varphi^r P_n^{(r)}}{}    \sim  n^{-r} \norm{w_n \varphi^r P_n^{(r)}}{}  \sim \norm{w_n \rho_n^r P_n^{(r)}}{}   \sim   \norm{w \rho_n^r P_n^{(r)}}{}
  \leq c   \norm{w P_n }{} \sim  \norm{w_n P_n }{} ,
\]
where the constant $c$ and all equivalence constants
are independent of $n$ and $P_n$.
\end{theorem}

\begin{proof}
The statement of the lemma is an immediate consequence of \cor{newjust1} and either of the estimates
\[
   \norm{w_n \rho_n^r P_n^{(r)} }{} \leq c \norm{w_n P_n}{} ,
\]
 (see \cite[Lemma 6.1]{k-acta}, for example), or
\[
   \norm{w  \varphi^r P_n^{(r)} }{} \leq c n^r \norm{w  P_n}{} ,
\]
(see \cite[(7.29)]{mt2000} or \cite[(2.5)]{mt2001}), where   the constant $c$ depends only on $r$ and the \astar{} constant of $w$.
\end{proof}

\begin{lemma} \label{lem8.5j}
Let $w$ be an \astar{} weight, $A >0$ and $\Z\in\bZM$.  Then   for any
$n, r\in\N$, $1\leq j \leq M$, and any polynomials $Q_n\in\Poly_n$ and $q_r \in \Poly_r$ satisfying  $Q_n^{(\nu)}(z_j)=q_r^{(\nu)}(z_j)$, $0\leq \nu\leq r-1$, the following inequality holds
\[
\norm{w(Q_n-q_r)}{\Z_{A, 1/n}^j}  \leq c n^{-r}     \norm{w \varphi^r  Q_n^{(r)} }{} ,
\]
where the constant $c$ depends only on $r$,  $A$ and the weight $w$.
\end{lemma}

\begin{proof}
Denote  $I := \Z_{A, 1/n}^j$,   $z:=z_j$, and note that $(Q_n - q_r)^{(\nu)}(z)=0$, $0\leq \nu\leq r-1$.
Using Taylor's theorem with the integral remainder 
we have
\[
Q_n(x) - q_r(x)  = {1 \over (r-1)!} \int_z^x  (x-u)^{r-1}  Q_n^{(r)}(u) du ,
\]
which implies
\begin{eqnarray*}
  \norm{w(Q_n-q_r)}{I}    & \leq &
  \sup_{x\in I} w(x)   \left|\int_z^x (x-u)^{r-1}   Q_n^{(r)}(u)    du\right|
  \leq
\norm{Q_n^{(r)}}{I} \sup_{x\in I}  w(x)|x-z|^r   \\
& \leq &
\left(A\rho_n(z)\right)^r \norm{Q_n^{(r)}}{I}  \sup_{x\in I}  w(x)
 \leq
c  \norm{\rho_n^r Q_n^{(r)}}{I} {1 \over |I|}   w(I)   ,
\end{eqnarray*}
 where, in the last inequality, we used the fact that $w$ is an \astar{} weight and
 $\rho_n(x)\sim \rho_n(z)$, $x\in I$. Now, since $w$ is doubling, $w(I)/|I| \leq c w[z-\rho_n(z), z+\rho_n(z)]/|I| \leq c   w_n(z) \leq c w_n(x)$, $x\in I$, and so
\[
  \norm{w(Q_n-q_r)}{I}    \leq
c  \norm{w_n \rho_n^r Q_n^{(r)}}{I}
  \leq
c  n^{-r} \norm{w  \varphi^r Q_n^{(r)}}{} ,
\]
 where the last estimate follows from \thm{thm5.5} .
\end{proof}

\begin{lemma} \label{lem7.2}
Let $\Z\in\bZ_M$,  $w\in \W^*(\Z)$, $c_*>0$,
  $n,r\in\N$, $A>0$ and $0<t\leq c_*/n$. Then, for any $P_n\in\Poly_n$, we have
\[
\Omega_\varphi^r(P_n, A, t)_{w} \leq c t^r \norm{w  \varphi^r P_n^{(r)}}{} ,
\]
where $c$ depends only on  $r$, $c_*$   and the weight $w$.
\end{lemma}

\begin{remark}
Using the same method as the one used to prove \cite{k-singular}*{Lemma 8.2} one can show that a stronger result than \lem{lem7.2} is valid. Namely, if
$f$ is such that $f^{(r-1)}\in \AC_\loc \left( (-1,1)\setminus \Z  \right)$
  and $\norm{w \varphi^r f^{(r)}}{} < \infty$, then one can show that
\[
\Omega_\varphi^r(f, A, t)_{w} \leq c t^r \norm{w  \varphi^r f^{(r)}}{} , \quad t>0.
\]
However,  \lem{lem7.2} whose proof is simpler and shorter is sufficient for our purposes.
\end{remark}

\begin{proof}[Proof of \lem{lem7.2}]
It follows from \cite{k-acta}*{Lemma 7.2} and \cor{newjust1} that, for any $c_*>0$ and $0<t \leq c_*/n$,
\[
\w_\varphi^r(P_n, t)_{w_n} \leq c t^r \norm{w_n \varphi^r P_n^{(r)}}{} \leq c t^r \norm{w  \varphi^r P_n^{(r)}}{},
\]
where the constants $c$ depend  on $r$, $c_*$ and the weight $w$.
Therefore, since any \astar{} weight $w$ satisfies $w(x) \leq c w_n(x)$, for any $x\in [-1,1]$ and $n\in\N$ (see \ineq{wlesswn}), we have
\begin{eqnarray*}
   \Omega_\varphi^r(P_n, A, t)_{w}   &= &  \sup_{0<h\leq t} \norm{w(\cdot) \Delta_{h\varphi(\cdot)}^r(P_n,\cdot, \I_{A,h})}{}
   \leq   c \sup_{0<h\leq t} \norm{w_n(\cdot) \Delta_{h\varphi(\cdot)}^r(P_n,\cdot, [-1,1])}{}\\
   & \leq & c \w_\varphi^r(P_n, t)_{w_n}\leq c t^r \norm{w \varphi^r P_n^{(r)}}{} .
\end{eqnarray*}
\end{proof}

\sect{Direct theorem}

\begin{theorem} \label{jacksonthm}
Let $w\in\W^*(\Z)$, $r,\nu_0\in\N$, $\nu_0\geq r$, $\ccc>0$, $f\in\L_\infty^w$ and $B>0$.
 Then, there exists $N\in\N$ depending on  $r$, $\ccc$ and the weight $w$, such that
 for every $n \geq N$,     there is a polynomial $P_n \in\Poly_n$ satisfying
\be \label{thing1}
\norm{w(f-P_n)}{}   \leq c   \w_\varphi^r(f, 1, B, \ccc/n)_{ w }
\ee
and
\be \label{thing2}
\norm{ w  \varphi^\nu P_n^{(\nu)}}{} \leq    c n^\nu \w_\varphi^r(f, 1, B, \ccc/n)_{ w }   ,      \quad r\leq \nu \leq \nu_0,
\ee
where constants $c$  depend only on  $r$, $\nu_0$,  $B$, $\ccc$   and the weight $w$.
\end{theorem}

We use an idea from \cite[Section 3.2]{mt2001} and deduce \thm{jacksonthm} from the following result that was proved in \cite{k-acta}.

\begin{theorem}[\mbox{\cite{k-acta}*{Theorem 5.3 ($p=\infty$)}}] \label{jacksonthmacta}
Let $w$ be a doubling weight, $r, \nu_0\in\N$, $\nu_0\geq r$, and $f\in\L_\infty[-1,1]$.
 Then, for every $n \geq r$ and $0<\ccc\leq 1$,  there exists a polynomial $P_n \in\Poly_n$ such that
\[
\norm{w_n(f-P_n)}{}    \leq c \w_\varphi^r(f, \ccc/n)_{w_n}
\]
and
\[
\norm{w_n \rho_n^\nu P_n^{(\nu)}}{}   \leq c \w_\varphi^r(f, \ccc/n)_{w_n} ,      \quad r\leq \nu \leq \nu_0,
\]
where constants $c$ depend only on  $r$, $\nu_0$,  $\ccc$ and the doubling constant of $w$.
\end{theorem}

\begin{proof}[Proof of \thm{jacksonthm}]
 Since $\w_\varphi^r(f, 1, B, t)_{ w }$ is a nondecreasing function of $t$, without loss of generality we can assume that with $\ccc \leq  1/(2r)$.
Suppose that $N\in\N$ is such that $N\geq \max\{ r, 100/(\ccc \D)\}$,   $n\geq N$, and
 and let $(x_i)_{i=0}^n$ be the Chebyshev partition of $[-1,1]$, \ie $x_i = \cos(i\pi/n)$, $0\leq i \leq n$ (for convenience,
we also denote $x_i :=-1$, $i\geq n+1$, and $x_i :=1$, $i\leq -1$).
As usual, we let $I_i := [x_i, x_{i-1}]$ for $1\leq i\leq n$.
Note that each (nonempty) interval $[z_j, z_{j+1}]$, $0\leq j\leq M$,   contains at least $10$ intervals $I_i$.

For each $1\leq j\leq M$, denote
\[
\nu_j := \min \left\{ i  \st  1\leq i\leq n \andd z_j \in I_i \right\} \andd  J_j:= [x_{\nu_j+1}, x_{\nu_j-2}] .
\]
Note that $\min$ in the definition of $\nu_j$ is needed if $z_j$ belongs to more than one (closed) interval $I_i$ (in which case $\nu_j$ is chosen so that $z_j$ is the left endpoint of $I_{\nu_j}$).
Let $q_j \in \Poly_r$ be a  polynomial of near best weighted approximation of $f$ on $J_j$, \ie
$\norm{w(f-q_j)}{J_j} \leq c E_r(f, J_j)_{w}$,  $1\leq j \leq M$,
and define
\[
F(x):= \begin{cases}
q_j(x) , & \mbox{\rm if }\; x  \in J_i  , \; 1\leq j \leq M ,\\
f(x), & \mbox{\rm otherwise.}
\end{cases}
\]

%
%
%

Since (see \cite{k-singular}*{p. 27}, for example)
$|I_i|/3 \leq |I_{i+1}| \leq 3 |I_i|$,  $1\leq i \leq n-1$, and
$\rho_n(x) \leq |I_i| \leq 5 \rho_n(x)$ for all $x\in I_i$ and $1\leq i \leq n$,
%
%
we conclude that
\begin{eqnarray*}
\max\{ |x_{\nu_j+1}-z_j|, |x_{\nu_j-2}-z_j|\} &\leq& \max\{ |I_{\nu_j+1}| + |I_{\nu_j}|,|I_{\nu_j}| + |I_{\nu_j-1}| \} \\
& \leq & 4 |I_{\nu_j}| \leq 20 \rho_n(z_j) \leq (20/\ccc^2) \rho(\ccc/n, z_j) ,
\end{eqnarray*}
and so
\[
J_j \subset \Z^j_{20/\ccc^2, \ccc/n}    , \quad 1\leq j \leq M.
\]
Therefore,
\begin{eqnarray} \label{Ff}
\norm{w(F-f)}{} & = &  \max_{1\leq j \leq M} \norm{w(q_j-f)}{J_i}
  \leq   c \max_{1\leq j \leq M} E_r(f, J_j)_{w} \\ \nonumber
   & \leq &   c \sum_{j=1}^M  E_r(f, \Z^j_{20/\ccc^2, \ccc/n} )_{w}
\leq  c \w_\varphi^r(f, 1, 20/\ccc^2, \ccc/n)_w .
\end{eqnarray}

We now estimate $\w_\varphi^r(F, \ccc/n)_{w_n}$ in terms of the  modulus of $f$.   Let $0<h\leq \ccc/n$ and $x$ such that $[x-rh\varphi(x)/2,x+rh\varphi(x)/2]\subset [-1,1]$ be fixed, and consider
the following three cases.

{\bf Case 1:}
$\ds x\in\Range_1 := \left\{ x \st [x-rh\varphi(x)/2,x+rh\varphi(x)/2] \subset J_j , \; \mbox{\rm for some }\;  1\leq j \leq M  \right\}$.\\
Then, for some $1\leq j \leq M$, $\Delta_{h \varphi(x)}^r(F, x, [-1,1]) = \Delta_{h \varphi(x)}^r(q_j, x, [-1,1]) = 0$, and so
\[
\norm{w_n(\cdot) \Delta_{h \varphi(\cdot)}^r(F, \cdot, [-1,1])}{\Range_1} = 0 .
\]

{\bf Case 2:}
  $\ds x\in\Range_2 := \left\{ x \st [x-rh\varphi(x)/2,x+rh\varphi(x)/2] \cap \cup_{j=1}^M J_i = \emptyset\right\}$. \\
Then, taking into account that
\begin{eqnarray} \label{newarr}
J_j &\supset&  [z_j - |I_{\nu_j+1}|, z_j + |I_{\nu_j-1}|]
\supset [z_j - |I_{\nu_j}|/3, z_j +|I_{\nu_j}|/3] \\ \nonumber
& \supset & [z_j - \rho_n(z_j)/3, z_j +\rho_n(z_j)/3] = \Z_{1/3, 1/n}^j ,
\end{eqnarray}
we conclude that
$x\in \I_{1/3, 1/n}$, and so $w_n(x) \sim w(x)$ by \lemp{v}.
Also, \ineq{newarr} implies that
\[
[-1,1] \setminus \cup_{j=1}^M J_j \subset [-1,1] \setminus \cup_{j=1}^M \Z_{1/3, 1/n}^j
\subset \I_{1/3, 1/n} \subset \I_{1/3, h} ,
\]
and so $[x-rh\varphi(x)/2,x+rh\varphi(x)/2] \subset \I_{1/3, h}$. Therefore,
$\Delta_{h \varphi(x)}^r(F, x, [-1,1]) = \Delta_{h \varphi(x)}^r(f, x, \I_{1/3, h})$, and
\[
\norm{w_n(\cdot) \Delta_{h \varphi(\cdot)}^r(F, \cdot, [-1,1])}{\Range_2} \leq c  \norm{w (\cdot)\Delta_{h \varphi(\cdot)}^r(f, \cdot, \I_{1/3, h})}{\Range_2}.
\]

{\bf Case 3:}
  $\ds x\in  \Range_3^j$, for some $1\leq j\leq M$, where  $\Range_3^j$ is the set of all $x$ such that $[x-rh\varphi(x)/2,x+rh\varphi(x)/2]$ has nonempty intersections with $J_j$ and $\left([-1,1]\setminus J_j\right)^{cl}$, \ie
  \[
  \Range_3^j := \left\{ x \st  x_{\nu_j+1}\; \mbox{\rm or }\; x_{\nu_j-2} \in [x-rh\varphi(x)/2,x+rh\varphi(x)/2] \right\} .
  \]
Note that, because of the restrictions on $N$, $[x-rh\varphi(x)/2,x+rh\varphi(x)/2]$ cannot have nonempty intersection with more than one interval $J_i$, and, in fact, $\Range_3^j$ is ``far'' from all intervals $J_i$ with $i\neq j$.

Without loss of generality, we can assume that $x_{\nu_j+1}  \in [x-rh\varphi(x)/2,x+rh\varphi(x)/2]$, since the other case follows by symmetry.
Taking into account that $x - rh\varphi(x)/2$ and $x + rh\varphi(x)/2$ are both increasing  functions in $x$, we have
\[
  \dist\left\{ z_j,  [x-rh\varphi(x)/2,x+rh\varphi(x)/2] \right\}    =
  z_j   - x - rh\varphi(x)/2
 \geq
 z_j -  \widetilde x -rh \varphi (\widetilde x)/2 ,
\]
where $\widetilde x$ is such that $\widetilde x - rh \varphi (\widetilde x)/2 = x_{\nu_j+1}$. Note that $\widetilde x <  x_{\nu_j}$ since
\[
x_{\nu_j} - rh \varphi (x_{\nu_j})/2 >  x_{\nu_j} - r\ccc \rho_n(x_{\nu_j})/2 \geq x_{\nu_j} - r\ccc |I_{\nu_j+1}|/2 > x_{\nu_j+1},
\]
and so $\widetilde x \in I_{\nu_j+1}$. Therefore,
\begin{eqnarray*}
\lefteqn{  \dist\left\{ z_j,  [x-rh\varphi(x)/2,x+rh\varphi(x)/2] \right\} }\\
&\geq&  z_j - x_{\nu_j+1} - rh \varphi (\widetilde x)
  \geq   |I_{\nu_j+1}| - r\ccc \rho_n(\widetilde x) \\
  & \geq & (1-r\ccc) |I_{\nu_j+1}| \geq (1-r\ccc)\rho_n(z_j)/3 = \rho_n(z_j)/6.
\end{eqnarray*}
Also,
\[
  \max\left\{|y-z_j| \st y\in  [x-rh\varphi(x)/2,x+rh\varphi(x)/2] \right\}
  =
  z_j   - x + rh\varphi(x)/2
\leq
 z_j -  \widehat x + rh \varphi (\widehat x)/2 ,
\]
where $\widehat x$ is such that $\widehat x + rh \varphi (\widehat x)/2 = x_{\nu_j+1}$.
Now, $\widehat x >  x_{\nu_j+2}$ since
\[
x_{\nu_j+2} + rh \varphi (x_{\nu_j+2})/2 <   x_{\nu_j+2} + r\ccc \rho_n(x_{\nu_j+2})/2 <  x_{\nu_j+2} +   r\ccc |I_{\nu_j+2}|/2 < x_{\nu_j+1} ,
\]
and so $\widehat x \in I_{\nu_j+2}$. Therefore,
\begin{eqnarray*}
\lefteqn{ \max\left\{|y-z_j| \st y\in  [x-rh\varphi(x)/2,x+rh\varphi(x)/2] \right\} }\\
 &\leq &
 z_j - x_{\nu_j+1} + rh \varphi (\widehat x) \leq  x_{\nu_j-1} - x_{\nu_j+1} + r\ccc \rho_n (\widehat x)\\
 & \leq &
 |I_{\nu_j+1}| +  |I_{\nu_j}| + r\ccc  |I_{\nu_j+2}| \leq (20+45 r\ccc) \rho_n(z_j) \leq 50 \rho_n(z_j).
\end{eqnarray*}
Hence,
\[
[x-rh\varphi(x)/2,x+rh\varphi(x)/2] \subset \I_{1/6, 1/n} \cap \Z^j_{50,1/n}   \subset \I_{1/6, h} \cap \Z^j_{50,1/n}     .
\]
\lemp{v} implies that  $w_n(x) \sim w(x)$. Also,
for any $y\in [x-rh\varphi(x)/2,x+rh\varphi(x)/2]$,
$|x-y| \leq rh\varphi(x)/2   \leq r\ccc \rho_n(x)/2 = \rho_n(x)/4$, and so \lemp{iv} yields  $w(y)\sim w(x)$.

This implies
\begin{eqnarray*}
\lefteqn{ \norm{w_n(\cdot) \Delta_{h \varphi(\cdot)}^r(F, \cdot, [-1,1])}{\Range_3^j}}\\
 & \leq  & \norm{w_n(\cdot) \Delta_{h \varphi(\cdot)}^r(f, \cdot, [-1,1]) }{\Range_3^j}
+ \norm{w_n(\cdot) \Delta_{h \varphi(\cdot)}^r(F-f, \cdot , [-1,1])}{\Range_3^j} \\
 & \leq  & \norm{w_n(\cdot) \Delta_{h \varphi(\cdot)}^r(f, \cdot, \I_{1/6, h}) }{\Range_3^j}
+ \norm{w_n(\cdot) \sum_{i=0}^r
{r \choose i}
  \left|(F-f)(\cdot-rh/2+ih\varphi(\cdot))\right|
}{\Range_3^j} \\
& \leq &
c \norm{w (\cdot) \Delta_{h \varphi(\cdot)}^r(f, \cdot, \I_{1/6, h}) }{\Range_3^j}
+ c \norm{w(q_j-f)}{J_j} \\
& \leq &
c \norm{w (\cdot) \Delta_{h \varphi(\cdot)}^r(f, \cdot, \I_{1/6, h}) }{\Range_3^j}
+ c  E_r(f, \Z^j_{20/\ccc^2, \ccc/n})_{w} .
\end{eqnarray*}

Combining the above cases we conclude that
\begin{eqnarray*}
 \w_\varphi^r(F, \ccc/n)_{w_n}
 &\leq&  c \sup_{0<h\leq \ccc/n} \norm{w (\cdot) \Delta_{h \varphi(\cdot)}^r(f, \cdot, \I_{1/6, h}) }{}
+ c \sum_{j=1}^M  E_r(f, \Z^j_{20/\ccc^2, \ccc/n})_{w}\\
& \leq &
c \w^r_\varphi(f, 1/6, 20/\ccc^2, \ccc/n)_w \leq c \w^r_\varphi(f, 1, 20/\ccc^2, \ccc/n)_w .
\end{eqnarray*}
We now recall that \thm{thm5.5} implies that
$\norm{ w  \varphi^\nu P_n^{(\nu)}}{} \leq   c n^\nu  \norm{w_n \rho_n^\nu P_n^{(\nu)}}{}$, and so applying \thm{jacksonthmacta} for the function $F$
as well as the fact that $w(x) \leq cw_n(x)$ (see \ineq{wlesswn})
we conclude that  \ineq{thing1} and \ineq{thing2} are proved with
$\w^r_\varphi(f, 1, 20/\ccc^2, \ccc/n)_w$ instead of $\w^r_\varphi(f, 1, B, \ccc/n)_w$ on the right-hand side.

Now, if $B\geq 20/\ccc^2$, then
$\w^r_\varphi(f, 1, 20/\ccc^2, \ccc/n)_w \leq \w^r_\varphi(f, 1, B, \ccc/n)_w$.
If $B<20/\ccc^2$, then, since $\ccc/n < \D/(80/\ccc^2)<1$,
\cor{cor2.14} implies that
$\w^r_\varphi(f, 1, 20/\ccc^2, \ccc/n)_w \leq c \w^r_\varphi(f, 1, 2^{-m}\cdot 20/\ccc^2, \ccc/n)_w \leq c \w^r_\varphi(f, 1, B, 1/n)_w$, where $m := \lceil \log_2(20/(B\ccc^2)\rceil \in\N$, and the constant $c$ depends only on $r$, $B$, $\ccc$ and the weight $w$.

The proof is now complete.
\end{proof}

 \sect{Inverse theorem}

\begin{theorem} \label{conversethm}
Suppose that $\Z\in\bZM$, $w\in\W^*(\Z)$,  $f\in\L_\infty^w$, $A,B>0$ and
 $n,r\in\N$.
Then
\[
 \w_\varphi^r(f, A, B, n^{-1})_{w}
  \leq  c n^{-r}   \sum_{k=1}^{n}    k^{r-1  }     E_{k}(f, [-1,1])_{w } ,
\]
where
the constant $c$ depends only on $r$,  $A$, $B$,   and the  weight $w$.
\end{theorem}

\begin{proof}
Let $P_n^* \in \Poly_n$ denote a polynomial of (near) best approximation to $f$ with weight $w$, \ie
\[
c \norm{w(f-P_n^*)}{} \leq   \inf_{P_n\in\Poly_n} \norm{w(f-P_n)}{} = E_{n}(f, [-1,1])_{w} .
\]
We let $N\in\N$ be such that $2^N \leq n < 2^{N+1}$.
To estimate $\Omega_\varphi^r(f, A, n^{-1})_{w}$,
using \lem{normestimate} we have
\begin{eqnarray*}
\Omega_\varphi^r(f, A, n^{-1})_{w} & \leq & \Omega_\varphi^r (f, A, 2^{-N})_{w} \\
& \leq & \Omega_\varphi^r (f - P_{2^N}^*, A, 2^{-N})_{ w} + \Omega_\varphi^r (P_{2^N}^*, A, 2^{-N})_{ w} \\
& \leq & c \norm{w(f - P_{2^N}^*)}{} + \Omega_\varphi^r (P_{2^N}^*, A, 2^{-N})_{w} \\
& \leq &
 c E_{2^N}(f, [-1,1])_{w} + \Omega_\varphi^r (P_{2^N}^*, A, 2^{-N})_{w}.
\end{eqnarray*}
Now, using
\be \label{decom}
P_{2^N}^* =  P_1^* + \sum_{i=0}^{N-1} (P_{2^{i+1}}^* - P_{2^{i}}^*)
\ee
as well as \lem{lem7.2}  we have
\begin{eqnarray*}
\Omega_\varphi^r (P_{2^N}^*, A, 2^{-N})_{ w}  &\leq &
 \sum_{i=0}^{N-1} \Omega_\varphi^r  \left( P_{2^{i+1}}^* - P_{2^{i}}^*, A, 2^{-N}\right)_{w }
 \leq
 c 2^{-Nr } \sum_{i=0}^{N-1}    \norm{ w \varphi^r \left(P_{2^{i+1}}^* - P_{2^{i}}^*\right)^{(r)}}{} .
\end{eqnarray*}
Now, for each $1\leq j\leq M$, taking into account that $\Z_{B,t_1}^j \subset \Z_{B,t_2}^j$ if $t_1 \leq t_2$, we have

\begin{eqnarray*}
E_r(f, \Z_{B,1/n}^j)_w   &\leq&  E_r(f, \Z_{B,2^{-N}}^j)_w
   \leq     \norm{w(f-P_{2^N}^*)}{\Z_{B,2^{-N}}^j} +
 E_r(P_{2^N}^*, \Z_{B,2^{-N}}^j)_w \\
 & \leq & c E_{2^N}(f, [-1,1])_{w} +   \norm{w(P_{2^N}^*-q_r(P_{2^N}^*)) }{\Z_{B,2^{-N}}^j},
\end{eqnarray*}
 where $q_r(g)$ denotes the Taylor polynomial of degree $<r$ at $z_j$ for $g$.
Using \ineq{decom} again,
    noting that
\be \label{extrataylor}
q_r(P_{2^N}^*) =  P_1^* + \sum_{i=0}^{N-1} q_r(P_{2^{i+1}}^* - P_{2^{i}}^*),
\ee
and taking \lem{lem8.5j} into account we have
\begin{eqnarray*}
\norm{w(P_{2^N}^*-q_r(P_{2^N}^*) )}{\Z_{B,2^{-N}}^j } & \leq &
\sum_{i=0}^{N-1} \norm{w\left( (P_{2^{i+1}}^* - P_{2^{i}}^*) - q_r(P_{2^{i+1}}^* - P_{2^{i}}^*)\right)}{\Z_{B,2^{-N}}^j} \\
& \leq &
 c \sum_{i=0}^{N-1}  2^{-Nr} \norm{w \varphi^r (P_{2^{i+1}}^* - P_{2^{i}}^*)^{(r)}}{} .
\end{eqnarray*}
Hence,
\[
 \w_\varphi^r(f, A, B, n^{-1})_{w} \leq c E_{2^N}(f, [-1,1])_{w} + c 2^{-Nr } \sum_{i=0}^{N-1}    \norm{w \varphi^r \left(P_{2^{i+1}}^* - P_{2^{i}}^*\right)^{(r)}}{}.
\]
Now, using    \thm{thm5.5} we have
\begin{eqnarray*}
 \w_\varphi^r(f, A, B, n^{-1})_{w} &\leq & c E_{2^N}(f, [-1,1])_{w} +
  c 2^{-Nr } \sum_{i=0}^{N-1}   2^{ir }  \norm{w( P_{2^{i+1}}^* - P_{2^{i}}^* ) }{}\\
 &\leq&
  c 2^{-Nr } \sum_{i=0}^{N}   2^{ir }    E_{2^i}(f, [-1,1])_{w} \\
& \leq &  c n^{-r}  \left(E_{1}(f, [-1,1])_{w} +  \sum_{i=1}^{N}  \sum_{k=2^{i-1}+1}^{2^{i}} k^{r-1 }     E_{k}(f, [-1,1])_{w} \right) \\
 & \leq &
c n^{-r}  \sum_{k=1}^{n}    k^{r-1}     E_{k}(f, [-1,1])_{w} ,
\end{eqnarray*}
with all constants $c$ depending only on $r$,  $A$, $B$, and the weight $w$.
\end{proof}

\sect{Realization functionals}

For $w\in\W^*(\Z)$, $r \in\N$,   and $f\in\L_\infty^w$, we
 define the following  ``realization functional''  as follows
 \[
R_{r,\varphi} (f, t,   \Poly_n)_{ w} := \inf_{P_n\in\Poly_n} \left( \norm{w(f-P_n)}{} + t^r \norm{w \varphi^r P_n^{(r)}}{} \right) ,
 \]
and note that $R_{r,\varphi} (f, t_1,   \Poly_n)_{w} \sim R_{r,\varphi} (f, t_2,   \Poly_n)_{w}$ if $t_1 \sim t_2$.

\begin{theorem} \label{corr99}
Let $\Z\in\bZM$, $w\in\W^*(\Z)$,  $f\in\L_\infty^w$, $A, B>0$,
  $r  \in\N$, and let $\ccc_2\geq \ccc_1>0$.
Then, there exists a constant $N\in\N$ depending only on $r$, $\ccc_1$, and the weight $w$, such that, for $n\geq N$ and  $\ccc_1/n \leq t \leq \ccc_2/n$,
\[
R_{r,\varphi} (f, 1/n, \Poly_n)_{ w}     \sim   \w_\varphi^r(f, A,  B, t)_{w } ,
\]
where the equivalence constants depend only on $r$, $A$, $B$, $\ccc_1$, $\ccc_2$  and the weight $w$.
\end{theorem}

\begin{proof} In view of \cor{corollary411} it is sufficient to prove this lemma for $A=1$.
 \thm{jacksonthm}  implies that, for every $n\geq N$ (with $N$ depending only on $r$, $\ccc_1$  and the weight $w$),
   there exists a polynomial $P_n \in\Poly_n$ such that
\be \label{kf1}
R_{r,\varphi} (f, 1/n, \Poly_n)_{w}   \leq c  \w_\varphi^r(f, 1,  B, \ccc_1/n)_{ w } \leq c \w_\varphi^r(f, 1,  B, t)_{w } .
\ee
%

Now,  let $P_n$ be an arbitrary polynomial from $\Poly_n$, $n\in\N$.
Lemmas~\ref{normestimate} and  \ref{lem7.2}  imply that
 \begin{eqnarray} \label{kf2}
 \Omega_\varphi^r(f, 1,  t)_{w}    &\leq&  c \Omega_\varphi^r(f-P_n, 1,  t)_{ w}  + c \Omega_\varphi^r(P_n, 1,  t)_{ w} \\ \nonumber
 & \leq &
 c \norm{w(f-P_n)}{} + c n^{-r} \norm{ w \varphi^r P_n^{(r)}}{} ,
\end{eqnarray}
where constants $c$ depend  only on  $r$,   $\ccc_2$  and the weight $w$.
Also,   taking into account that $\Z_{B,t}^j \subset \Z_{B, \ccc_2/n}^j \subset \Z_{B \ccc_2\max\{\ccc_2,1\}, 1/n}^j$ and using \lem{lem8.5j}, we have
\begin{eqnarray} \label{kf4}
\sum_{j=1}^M E_r (f, \Z_{B,t}^j)_w
 & \leq & c \norm{w(f-P_n)}{}   + \sum_{j=1}^M \inf_{q\in\Poly_r} \norm{w(P_n - q)}{\Z_{B \ccc_2\max\{\ccc_2,1\}, 1/n}^j}\\ \nonumber
 & \leq &
 c \norm{w(f-P_n)}{} +    c n^{-r}  \norm{w \varphi^r  P_n^{(r)} }{}.
\end{eqnarray}

Therefore, for any $n\in\N$, $\ccc_2>0$ and $0<t\leq \ccc_2/n$,
\be \label{kf99}
\w_\varphi^r(f, 1,  B, t)_{ w } \leq c R_{r,\varphi} (f, 1/n, \Poly_n)_{w} ,
\ee
which completes the proof of the theorem.
\end{proof}

 \thm{corr99} implies, in particular,  that
 $\w_\varphi^r(f, A_1 , B_1, t_1)_{w } \sim \w_\varphi^r(f, A_2, B_2, t_2)_{w }$ if   $t_1 \sim t_2$ with equivalence constants independent of $f$.

 Finally, we remark that the moduli $\w_\varphi^r(f, A , B, t)_{w }$ are not equivalent to the following weighted $K$-functional
 \[
 K_{r,\varphi} (f, t)_{ w} := \inf_{g^{(r-1)}\in\AC_\loc} \left( \norm{w(f-g)}{} + t^r \norm{w \varphi^r g^{(r)}}{} \right) .
 \]
 This follows from counterexamples constructed in \cite{mt1999}, where additional discussions and  negative results can be found.

\sect{Appendix}

The following lemma shows that $E_r(f, \Z_{B,t}^j)_w$ in the definition of the complete modulus \ineq{compmod} can be replaced with
$\norm{w(f-q_j)}{\Z_{B,t}^j}$, where $q_j$ is a polynomial of (near) best weighted approximation to $f$ on any subinterval of $\Z_{B,t}^j$ of length $\geq c \rho(t, z_j)$.

 \begin{lemma} \label{lem3.1}
 Suppose that  $\Z\in\bZM$, $w\in\W^*(\Z)$,  $f\in\L_\infty^w$,
 and suppose that  intervals $I$ and $J$ are such that $I\subset J \subset [-1,1]$ and $|J|\leq c_0|I|$.
 Then, for any $r\in\N$,
if   $q \in\Poly_r$ is a polynomial of near best approximation to $f$ on $I$  with weight $w$, \ie
  \[
\norm{w(f-q)}{I} \leq c_1 E_r(f,I)_{w} ,
\]
then $q$ is also a polynomial of near best approximation to $f$ on $J$. In other words,
  \[
\norm{w(f-q)}{J} \leq c E_r(f,J)_{w} ,
\]
where the constant $c$ depends only on  $r$, $c_0$, $c_1$  and the weight $w$.
 \end{lemma}

\begin{proof} The proof is similar to that of \cite{k-singular}*{Lemma A.1}.
First, we assume that $|I| \leq \D/2$, and so $I$ may contain at most one $z_j$ from $\Z$.
Now, we denote by $a$  the midpoint of $I$ and
 let $n\in\N$ be such that
$\rho_{n+1}(a) < |I|/1000 \leq \rho_n(a)$. Then, $\rho_n(a) \sim |I|$ and, as was shown in the proof of \cite{k-singular}*{Lemma A.1},
 $I$ contains at least $5$ adjacent  intervals $I_{\nu+i}$, $i=2,1,0,-1,-2$.
 Moreover,  one of those intervals, $I_\mu$, is such that $|I_\mu|\sim |I|$ and $I_\mu \subset \I_{c, 1/n}$ with some absolute constant $c$, and \lemp{iv} implies that $w(x) \sim w (y)$, for $x,y\in I_\mu$, with equivalence constants depending only on $w$.

Suppose now  that $\widetilde q$ is a polynomial of near best weighted approximation of $f$ on $J$, \ie
$\norm{w(f-\widetilde q)}{J} \leq c E_r(f, J)_{w}$.
 Then, taking into account that $|I_\mu| \sim |I| \sim |J|$ and using the fact that   $w$ is doubling,
 we have
\begin{eqnarray*}
\norm{w(\widetilde q -q)}{J}   & \leq &   L^* |J|^{-1}  \norm{\widetilde q - q}{J} w(J)
  \leq
c |I_\mu|^{-1} \norm{\widetilde q - q}{I_\mu} w(I_\mu) \\
 & \leq &
c w(x_\mu) \norm{\widetilde q - q}{I_\mu}
\leq c \norm{w(\widetilde q -q)}{I_\mu} .
\end{eqnarray*}

Therefore,
\begin{eqnarray*}
\norm{w(f-q)}{J} &\leq& c \norm{w(f-\widetilde q)}{J}+ c \norm{w(\widetilde q -q)}{J} \\
& \leq &
c \norm{w(f-\widetilde q)}{J} + c \norm{w(\widetilde q -q)}{I} \\
& \leq &
c \norm{w(f-\widetilde q)}{J} + c \norm{w(\widetilde q -f)}{I}    +  c \norm{ w( f -q)}{I}  \\
& \leq &
c \norm{w(f-\widetilde q)}{J}     +  c \norm{w( f -q)}{I}  \\
& \leq &
c E_r(f, J)_{w} + c E_r(f, I)_{ w} \\
& \leq &
c E_r(f, J)_{ w} ,
\end{eqnarray*}
 and the proof is complete if $|I| \leq \D/2$.

If $|I| > \D/2$, then $|I|\sim |J|\sim 1$, and we take $n\in\N$ to be such that $I$ contains at least $4M+4$ intervals $I_i$. Then $I$ contains   $4$ adjacent intervals $I_i$ not containing any points from $\Z$, and we can use the same argument as above.
\end{proof}

\end{document}